\documentclass[12pt]{article}
\usepackage{mathtools}
\usepackage{appendix}

\usepackage{amsmath,latexsym,amsfonts,enumerate}
\usepackage{amssymb,amsthm}
\usepackage{calrsfs} 
\DeclareMathAlphabet{\pazocal}{OMS}{zplm}{m}{n}
\usepackage{tikz-cd}
\usepackage{tikz}
\usepackage{comment}

\usepackage[LGR,T1]{fontenc} 
\usepackage[utf8]{inputenc} 
\usepackage{lmodern} 
\usepackage[greek,english]{babel} 
\usepackage{alphabeta} 
\usepackage{biblatex}
\usepackage{csquotes}
\usepackage{cleveref}
\usepackage{pgfplots}
\usepackage{tikz-3dplot}

\tdplotsetmaincoords{60}{115}
\pgfplotsset{compat=newest}

\newtheorem{theor}{Theorem}[section]
\newtheorem{lem}[theor]{Lemma}
\newtheorem{prop}[theor]{Proposition}
\newtheorem{cor}[theor]{Corollary}

\newtheorem{notation}[theor]{Notation}
\newtheorem{rem}[theor]{Remark}

\def\N{\mathbb{N}}
\def\R{\mathbb{R}}

\DeclareMathOperator{\dis}{{\mathbf d}}

\DeclareMathOperator{\Ran}{Ran}
\DeclareMathOperator{\Span}{span}

\DeclareMathOperator{\Trun}{\downarrow}

\DeclareMathOperator{\rank}{rank}

\DeclareMathOperator{\Lip}{Lip}

\newcommand{\ignore}[1]{}

\newcommand{\ip}[2]{{\langle#1,#2\rangle}}
\DeclarePairedDelimiter{\norm}{\lVert}{\rVert}

\newcommand{\Vc}{{\mathcal{V}}}
\newcommand{\bw}{{\textbf{w}}}

\newcommand{\wb}{{\mathbf w}}
\newcommand{\Vh}{{\widehat{V}}}
\newcommand{\dd}{\dis}

\title{G-Invariant Representations using Coorbits: Bi-Lipschitz Properties}
\author{{\bf Radu Balan, Efstratios Tsoukanis} \\
Department of Mathematics \\ University of Maryland \\ College Park, MD 20742 \\email: rvbalan@umd.edu efstratios.tsoukanis@cgu.edu}

\begin{document}

\maketitle

\begin{abstract}
	Consider a finite dimensional real vector space and a finite group acting unitarily on it. We study the general problem of constructing Euclidean stable embeddings of the quotient space of orbits. Our embedding is based on subsets of sorted coorbits. Our main result shows that, whenever such embeddings are injective, they are automatically bi-Lipschitz. Additionally, we demonstrate that stable embeddings can be achieved with reduced dimensionality, and that any continuous or Lipschitz $G$-invariant map can be factorized through these embeddings.

\end{abstract}
\section{Introduction}

In a lot of machine learning problems we want to embed our data into an Euclidean space $\mathbb{R}^m$ using a symmetry-invariant embedding $\Psi$ and utilize $\mathbb{R}^m$ as our feature space. This embedding $\Psi$ should also separate data orbits and satisfy  certain stability conditions to ensure that small perturbations of the input do not significantly impact the predictions. We worked in the orbit separation problem in \cite{balan2023inj}. In this paper we focus on the stability problem.

This problem is an instance of \emph{invariant machine learning}.
\cite{mixon2022injectivity,aslan2023group,dym2022low,bronstein,maron2018invariant,DUFRESNE20091979,cahill2023bilipschitz,derksen2024bi}. 

The most common group actions in invariant machine learning are permutations
\cite{sannai2020universal,cahill22,balan2022permutation}
reflections \cite{mixon2022max} and  translations \cite{Cahill19}.
A related, by somewhat different problem is the case of equivariant embeddings \cite{lipman2022,maron2018invariant,villar2021scalars,villar2022dimensionless,blum2022equivariant}.

The phase retrieval that was introduced in \cite{balan2004signal} is an instance of this setup, when the compact group is the torus. \cite{balwan,balan16,balan06,alaifari2019stable,freeman2024stable,alaifari2021stability}.

Our work extends and unifies two prior approaches: the \emph{max filter} embedding introduced in  \cite{cahill22,mixon2022injectivity}, and the \emph{permutation invariant representation} introduced in \cite{balan2022permutation}. 
 \cite{amir2025stability} analyzes the stability of generalized phase retrieval problems under the action of compact groups, providing conditions under which stable recovery is possible from invariant measurements.
The construction of permutation invariant embeddings is  closely connected to the phase retrieval problem \cite{balan2004signal,bcmn} that has a large body of results. 
For instance, \cite{balan16} provides exact estimates for both the upper and lower Lipschitz bounds, addressing both the real and complex cases of the phase retrieval problem. In \cite{Balan2023Rel}  we establish an isometric identification of the real phase retrieval problem to $S_2$-invariant representations.

A completely different approach is considered in \cite{eriksson2018quantitative}. There it is proved that for any discrete group $G$ of isometries acting on $\mathbb{R}^d$, it is possible to construct a bi-Lipschitz map from $\mathbb{R}^d/G$ to $\mathbb{R}^N$, where the distortion depends solely on $d$. 

Since the first draft of this paper was placed on arxiv, the authors of \cite{qaddura2024stable} extended the construction of this paper and proved that for certain compact groups, given enough generic templates, the coorbit filter bank (as defined there) is injective and bi-Lipschitz. 

In this paper we construct an Euclidean embedding that is globally bi-Lipschitz and can be implemented relatively easy.
Our paper is organized as follows:
Chapter 1 introduces the embedding map. Chapter 2 explores the upper Lipschitz bound for the proposed embeddings and establishes the equivalence between injectivity on the quotient space and stability. Chapter 3 demonstrates that a generic linear projection can be used to reduce the dimension of the target space while preserving both injectivity and stability. 
Finally, chapter 4 presents universal factorization results for arbitrary continuous or Lipschitz $G$-invariant maps.

\subsection{Notation}
In this paper we use the same notation as in \cite{balan2023inj}.
Let $(\Vc,\ip{\cdot}{\cdot} )$ be a $d$-dimensional real vector space with scalar product, and $d\geq 2$.
Assume $(G, \cdot)$ is a finite group of order $|G|=N$ acting unitarily on $\Vc$.
For every $g \in G$, we denote by $U_gx$ the group action on vector $x\in\Vc$.
Let $\hat{\Vc}=\Vc/\sim$ denote the quotient space with respect to this group action,  $x\sim y$ if and only if $y=U_gx$ for some $g\in G$. 
 We denote by $[x]$ the orbit of vector $x$, i.e. $[x]= \{U_gx :g \in G\}$. 
 The natural metric, $\dis: \hat{\Vc} \times \hat{\Vc} \to \R$, is defined by
\begin{equation}    
\dis([x],[y]) = \min_{h_1,h_2 \in G} \norm{U_{h_1}x-U_{h_2}y} = \min_{g \in G} \norm{x-U_gy}.
\end{equation}
Note $(\hat{\Vc},\dis)$ is a complete metric space.

Our goal is to construct a bi-Lipschitz Euclidean embedding of the metric space $(\hat{\Vc},\dis)$ into an Euclidean space $\R^m$.

Specifically, we want to construct a function $\Psi: \Vc \to \R^m$ such that
\begin{enumerate}
	\item  $\Psi(U_gx)=\Psi(x),\ \forall x \in \Vc,\ \forall g \in G$,\label{property1}
	\item  If $x,y \in \Vc$ are such that $\Psi(x)=\Psi(y)$, then there exist $g \in G$ such that $y=U_gx$,\label{property2}
	\item  There are $0<a<b<\infty$ such that for any $x,y \in \Vc$
	\[a \dis([x],[y])^2 \le\norm{\Psi(x)-\Psi(y)}^2\le b (\dis([x],[y]))^2.\]\label{property3}
\end{enumerate}   

The invariance property \eqref{property1} allows to lift $\Psi$ to a map $\hat{\Psi}$ acting on the quotient space $\hat{\Vc}=\Vc/\sim$:
\[ \hat{\Psi}:\hat{\Vc}\rightarrow\R^m,\quad \hat{\Psi}([x])=\Psi(x), \quad\forall [x]\in\hat{\Vc}. \]
If a $G$-invariant map $\Psi$ satisfies property \eqref{property2} we say that $\Psi$ {\em separates} $G$-orbits in $\Vc$.

Our construction for the embedding $\Psi$ is based on the non-linear sorting map $\Trun$ described next.

\begin{notation}
	$\Trun:\R^r \to \R^r$ denoted the operator that takes as input a vector in $\R^r$ and returns a monotonically decreasing sorted  vector of same length $r$ that has same entries as the input vector:
 \[ x\in\R^r \mapsto \Trun x = (x_{\sigma(i)})_{1\leq i\leq r}~~,~~x_{\sigma(1)}\geq\cdots\geq x_{\sigma(r)} \]
 for some permutation $\sigma:\{1,\ldots,r\}\rightarrow\{1,\ldots,r\}$.
\end{notation}

For an integer $p \in \N$, we denote $[p]=\{1,2,\ldots,p\}$. For a set $S$, $|S|$ denotes its cardinal. Fix a $p$-tuple of vectors $\bw= (w_1,\dots,w_p)\in \Vc^p$. For any  $i \in [p]$ and $j \in [N]$ we define the operator $\Phi_{w_i,j}: \Vc \to \R$ so that $\Phi_{w_i,j}(x)$ is the $j$-th coordinate of the sorted vector $\Trun(\ip{U_gw_i}{x})_{g\in G }$.
Fix a set $S \subset [N] \times [p]$ such that $|S|=m$, and for $i \in [p]$, let $S_i=\{k \in [N]: (k,i)\in S\}$ (the $i^{th}$ column of $S$). Denote by $m_i$ the cardinal of the set $S_i$, $m_i=|S_i|$. Thus $m= \sum_{i=1}^p m_i$.  
\begin{notation}
    The coorbit embedding $\Phi_{\bw,S}$ associated to windows $\bw\in\Vc^p$ and index set $S\subset [N] \times [p]$ is given by the map
\begin{equation}
    \label{emb}
\Phi_{\bw,S}:\Vc\rightarrow\R^m~,~
\Phi_{\bw,S}(x)=[\{\Phi_{w_1,j}(x)\}_{j \in S_1},\dots,\{\Phi_{w_p,j}(x)\}_{j \in S_p}] \in \R^m.
\end{equation}
\end{notation}

\begin{figure}

    \begin{tikzpicture}[]
    \draw[black, very thick] (0,0) rectangle (4,4)[label=right: {$\Vc=\R^d$}] {}; 
    \draw[black, very thick] (6,0) rectangle (9,4);
    \draw (2,1.8) node (yaxis) [below] {$[y]$};
     \draw  (2,3.3) node (yaxis) [below] {$[x]$};
     \draw  (2,-0.3) node (yaxis) [below] {$\hat{\Vc}$};
     \draw  (7.5,-0.3) node (yaxis) [below] {$\R^m$};
     \draw  (12.5,-0.3) node (yaxis) [below] {$\R^{q}$};
     \draw  (5,2.9) node (yaxis) [below] {$\Phi_{\bw,S}$};
     \draw  (5,1.9) node (yaxis) [below] {$\Phi_{\bw,S}$};
     \draw  (9.8,3.23) node (yaxis) [below] {$\ell$};
     \draw  (9.8,2.27) node (yaxis) [below] {$\ell$};
      \draw  (7.8,3.1) node (yaxis) [above] {$\Phi_{\bw,S}([x])$};
     \draw  (7,2.1) node (yaxis) [above] {$\Phi_{\bw,S}([y])$};
      \draw  (11.9,1) node (yaxis) [above] {$\ell \circ \Phi_{\bw,S}([y])$};
     \draw  (11.8,2.7) node (yaxis) [above] {$\ell \circ \Phi_{\bw,S}([x])$};
     \draw[black, very thick] (10.3,0) rectangle (13.5,4);
 \draw[red, thick, dotted, scale=1,domain=0:3,smooth]
  plot[parametric,id=parametric-example] function{sqrt(9-t*t),t};
 \draw[green, thick, dotted, scale=1,domain=0:4,smooth]
  plot[parametric,id=parametric-example] function{sqrt(16-t*t),t};
    \foreach \Point in {(8,3)  }{
    \node[green] at \Point {\textbullet};}
      \foreach \Point in {(7,2)  }{
    \node[red] at \Point {\textbullet};}
    \draw[dashed,->] (2.77,3) -- (7.8,3) [label=right: {$\Vc=\R^d$}] {};;
    \draw[dashed,->] (2.23,2) -- (6.8,2);
    \foreach \Point in {(11.8,2.6)  }{
    \node[green] at \Point {\textbullet};}
      \foreach \Point in {(12,1.6)  }{
    \node[red] at \Point {\textbullet};}
     \draw[dashed,->] (8.1,3) --(11.6,2.6) ;
    \draw[dashed,->] (7.1,2) --(11.9,1.63)  ;
    \end{tikzpicture}
\caption{Proposed embedding scheme}
\label{fig1}
\end{figure}

Let $\ell: \R^{m} \to \R^{q}$  be a linear map.
\begin{notation}
The embedding $\Psi_{\bw,S,\ell}$ associated to windows $\bw\in\Vc^p$, index set $S\subset [N] \times [p]$ and linear map $\ell:\R^m\rightarrow\R^q$ is given by the map
\begin{equation}
 \Psi_{\bw,S,\ell}=\ell \circ \Phi_{\bw,S} :\Vc \to \R^{q}~,~\Psi_{\bw,S,\ell}(x)=\ell(\Phi_{\bw,S}(x))
\end{equation}
obtained by composition of $\ell$ with the coorbit embedding $\Phi_{\bw,S}$.   
\end{notation}

 In this paper we focus on stability and universality properties of maps $\Phi_{\bw,S}$ and $\Psi_{\bw,S,\ell}$. 

Informally, our main results, \Cref{theormainstab}, 
 \Cref{dimred}, \Cref{LIPTHEOR} and \Cref{cont}, state that: (1) "injectivity" implies "(bi-Lipschitz) stability", (2) stable bi-Lipschitz embedding can be achieved into an Euclidean space of dimension at most twice the dimension of the input space; and (3) any continuous or even Lipschitz map factors through $\Phi_{\bw,S}$. 

\vspace{5mm}

For the rest of the paper we shall use interchangeably $\Phi_{i,j}$ instead of $\Phi_{w_i,j}$. 



\subsection{Main results}

In this section, we summarize the key results of this paper, that focus on bi-Lipschitz stability of embeddings, dimensionality reduction, and universality of \( G \)-invariant maps.

\begin{theor}\label{theormainstab}
	Let $G$ be a finite group acting unitarily on the vector space $\Vc$.  For fixed $\bw \in \Vc^p$ and $S\subset [N]\times [p]$, where $|S|=m$, suppose that the map $\hat{\Phi}_{\bw,S}: \hat{\Vc} \to \R^m$ is injective on the quotient space $\hat{\Vc}=\Vc/G$. Then, there exist $0<a \le b<\infty$ such that for all $x,y \in \Vc$, 
 \[a\dis([x],[y]) \le \norm{\Phi_{\bw,S}(x)-\Phi_{\bw,S}(y)}_{2}\le b\dis([x],[y]).\]
\end{theor}

\begin{cor}\label{maxfilter} When injective, the Max Filter embedding is bi-Lipschitz. Specifically:
Let \( G \) be a finite group acting unitarily on \( V \). Suppose the \text{max filter bank embedding} \( \hat{\Phi}_{w,S_{\text{max}}} : \hat{\Vc}\to \mathbb{R}^p \), defined by
\[
\hat{\Phi}_{w,S_{\text{max}}}([x]) = \left( \max_{g \in G} \langle x, g \cdot w_k \rangle \right)_{k \in [p]} 
\]
is injective. Then \( \hat{\Phi}_{w,S_{\text{max}}} \) is \textbf{bi-Lipschitz}, that is, there are $ 0 < a \leq b < \infty$, so that for every $x,y\in\Vc$,
\[
a \cdot \dis([x], [y]) \leq \| \Phi_{w,S_{\text{max}}}(x) - \Phi_{w,S_{\text{max}}}(y) \|_2 \leq b \cdot \dis([x], [y]).
\]
\end{cor}

\begin{theor}\label{dimred}
Let $\{U_g~,~g\in G\}$ denote a representation of a finite group $G$ of order $N$
 acting by isometries on the real vector space $\Vc$ of dimension $d$. 
 Let $V_G=\{x \in \Vc : U_gx=x ~,~\forall g \in G\}$ denote the linear space of vectors invariant of this representation, and let $d_G= \dim(V_G)$ denote its dimension. 
 Let $\bw\in\Vc^p$ and $S\subset[N]\times[p]$ so that $\hat{\Phi}_{\bw,S}: \hat{\Vc} \to \R^{m}$ is injective on the quotient space $\hat{\Vc}$. 
Then, for a generic linear map $\ell:\R^m \to \R^{2d-d_G}$, the map $\hat{\Psi}_{\bw,S,\ell}=\ell \circ \hat{\Phi}_{\bw,S}:\hat{\Vc}\rightarrow\R^{2d-d_G}$ is injective and bi-Lipschitz. Here generic means open dense with respect to Zariski topology over the set of matrices.
\end{theor}


\begin{theor}\label{LIPTHEOR}
Assume that for some $S \in [N] \times [p]$ and $\bw=(w_1,\dots,w_p) \in \Vc^p$, the map $\hat{\Phi}_{\bw,S}:\hat{\Vc} \to \R^m$ is bi-Lipschitz on the quotient space $\hat{\Vc}=\Vc/G$, with (upper) Lipschitz constant $b=Lip(\hat{\Phi}_{\bw,S})$ and lower Lipschitz constant $a>0$.
\begin{enumerate}
    \item For every Lipschitz map $F: \hat{\Vc} \to H$ into a real Hilbert space $H$ there exists a Lipschitz map $ T :\R^m \rightarrow H$ such that $F=T \circ \Phi_{\bw,S}$ and has Lipschitz constant $\Lip(T) \le \frac{1}{a} \Lip(F) $.
    \item  Conversely, for any Lipschitz map $T: \R^m \to H$, the map $F=T \circ \Phi_{\bw,S}$ is Lipschitz, with Lipschitz constant $\Lip(F) \le b \Lip(T) $.
\end{enumerate} 
\end{theor}

\begin{theor}\label{cont}
Assume that, for fixed $S \in [N] \times [p]$ and $\bw=(w_1,\dots,w_p) \in \Vc^p$, the map $\hat{\Phi}_{\bw,S}:\Vc \to \R^m$ is bi-Lipschitz on the quotient space $\hat{\Vc}=\Vc/G$. Let $L$ denote a locally convex topological vector space.
For every continuous map $F: \hat{\Vc} \to L$ there exists a continuous $T :\R^m \to L$  such that $F=T \circ \Phi_{\bw,S}$. Furthermore $T(\R^m)$ is included in the convex hull of $F(\hat{\Vc})$.
\end{theor}

\begin{rem}
    In \cite{balan2024stability} we generalize \cref{theormainstab} for the case where our embedding has the form $a \circ \Phi_{\bw,S}(x)$, where $a$ is a linear map.
\end{rem}

\begin{rem}
        In \cite{balan2023g} we study under which conditions the map  $\hat{\Phi}_{\bw,S}: \hat{\Vc} \to \R^m$ is injective on the quotient space $\hat{\Vc}=\Vc/G$.
\end{rem}

\section{Stability of Embedding}
Suppose that for $\bw=(w_1,\dots,w_p) \in \Vc^p$ and $S \subset [N] \times [p]$ the map $\hat{\Phi}_{\bw,S}$ is injective. In this case, we claim the map $\hat{\Phi}_{\bw,S}$ is also bi-Lipschitz, (\Cref{theormainstab}). An estimate of the upper Lipschitz constant $b$ is given in \Cref{lemupper}.

\Cref{maxfilter} follows from \Cref{theormainstab}.

The proof of \Cref{theormainstab} is contained in the next two subsections.

\subsection{Upper Lipschitz bound}

Note that $\Phi_{\wb,S}$ is Lipschitz because it is a composition of Lipschitz maps. It remains to estimate the upper bound.
\begin{lem}\label{lemupper}
  Consider $G$ be a finite group of size $N$ acting unitarily on  $\Vc$. Let $\wb \in \Vc^p$ and $S\subset  [N]\times [p]$.
Denote
\[  B = \max_{\substack{
\sigma_1,\ldots,\sigma_p\subset G \\
|\sigma_i|=m_i,\forall i
}}\lambda_{max}\left(\sum_{i=1}^p \sum_{g \in \sigma_i} U_g w_i w_i^T U_g^T  \right), \]
where $S_i=\{j\in[N],(i,j)\in S\}$ and $m_i=|S_i|$. 
Then $\hat{\Phi}_{\wb,S}:(\Vh,\dd)\rightarrow \R^m$ is Lipschitz with an upper Lipschitz constant $\sqrt{B}$.
\end{lem}

\begin{proof}
For fixed $x,y \in \Vc$, $i \in [p], j \in S_i$, let $\psi_{i,j}:[0,1] \to \R$, where \[\psi_{i,j}(t) =\Phi_{i,j}((1-t)x +ty) = \ip{(1-t)x +ty}{U_{g(t)}w_i}.\]

By Lebesgue differentiation theorem we have that $\psi_{i,j}$ is differentiable almost everywhere. Consequently $\Phi_{i,j}$ is also differentiable almost everywhere. Notice that
\begin{align*}
    \frac{d}{dt}\psi_{i,j}(t) = \ip{y-x}{U_{g(t)}w_i}
\end{align*}
for almost every $t$. Specifically, $\psi_{i,j}(t)$ is differentiable at all $t \in [0,1]$ such that there exists $\epsilon >0$ so that $g{\vert}_{(t-\epsilon,t+\epsilon)}$ can be chosen to be constant. This happens because $G$ is finite.

By the fundamental theorem of calculus we get \[\Phi_{i,j}(x)-\Phi_{i,j}(y)= \int_0^1\frac{d}{dt}\Phi_{i,j}((1-t)x +ty)dt.\]

Therefore,

 \[\Phi_{i,j}(y)-\Phi_{i,j}(x)= \int_0^1\ip{y-x}{U_{g_{j_t}}w_i}dt\]
   so
   \begin{align*}
    \norm{\Trun \{\Phi_{i,j}(x)\}_{j \in S_i}-\Trun \{\Phi_{i,j}(y)\}_{j \in S_i}} &\le \int_0^1(\sum_{j \in S_i}\ip{y-x}{U_{g_{j_t}}w_i}^2)^{1/2}dt\\ &\le
    \sqrt{B_i} \norm{x-y}
    \end{align*}
where $B_i=\max_{\substack{
\sigma \subset G, \\
|\sigma|=|S_i|,
}}\lambda_{max}\left(\sum_{g \in \sigma} U_g w_i w_i^T U_g^T  \right).$

Hence, for $x \in [x], y \in [y]$ so that $\dis([x],[y])=\norm{x-y}$,
\begin{align*}
    \norm{\Phi_{\bw,S}(x)-\Phi_{\bw,S}(y)}^2= \sum_{i=1}^p \norm{\Phi_{i,j}(y)-\Phi_{i,j}(x)}^2 &\le \sum_{i=1}^p B_k \norm{x-y}^2\\ &\le B\dis([x],[y])^2.
\end{align*}

\end{proof}

\subsection{Lower Lipschitz bound}

We start with some useful geometric results.

\subsubsection{Geometric Analysis of Coorbits}
Fist let us introduce some additional notation.
For fixed $i \in [p]$, $j \in [N]$ and $x \in \Vc$ we define the following non-empty subset of the group $G$:
\begin{equation}\label{LIJ}
L^{i,j}(x)=\{g \in G: \ip{x}{U_g w_i} =\Phi_{i,j}(x)\}.
\end{equation}
This represents the collection of group elements that achieve the $j$-th position for the sorted co-orbit $\downarrow (\ip{U_g w_i}{x})_{g\in G}$.

Consider also the map

\begin{equation}\label{DIJ}
\Delta^{i,j}(x) =\begin{cases} \min_{g \notin L^{i,j}(x)}(|\ip{U_g w_i}{x}-\Phi_{i,j}(x)|)\frac{1}{\norm{w_i}},~\text{ if } L^{i,j}(x) \not= G\\ \frac{\norm{x}}{\norm{w_i}},~  \text{ if } L^{i,j}(x) = G. \\
\end{cases}
\end{equation}

\begin{lem}\label{lemPhi}
\mbox{}

a. For any $x\in\Vc$, $i\in[p]$, and $j\in[N]$, 
\begin{equation}
\label{eq:ineq1}    
\left|\{g\in G~,~\ip{U_g w_i}{x}>\Phi_{i,j}(x) \} \right| \leq j-1
\end{equation} 
\begin{equation}
\label{eq:ineq2}
\left|\{g\in G~,~\ip{U_g w_i}{x}<\Phi_{i,j}(x) \} \right| \leq N-j. 
\end{equation}

b. For any $x\in\Vc$, $i\in[p]$, and $j\in[N-1]$,

(i) either $\Phi_{i,j}(x)=\Phi_{i,j+1}(x)$, in which case $L^{i,j}(x)=L^{i,j+1}(x)$, 

or 

(ii)  $\Phi_{i,j}(x)>\Phi_{i,j+1}(x)$, in which case $L^{i,j}(x)\neq L^{i,j+1}(x)$ and  
\[ \{g\in G~,~\ip{U_g w_i}{x}>\Phi_{i,j+1}(x) \} = \cup_{k\leq j} L^{i,k}(x)
~~,~~
\left| \cup_{k\leq j} L^{i,k}(x)\right| = j   \]
and
\[ \{g\in G~,~\ip{U_g w_i}{x}<\Phi_{i,j}(x) \} =
\cup_{k\geq j+1} L^{i,k}(x) 
~~,~~
\left| \cup_{k\geq j+1} L^{i,k}(x)\right| = N- j.   \]
\end{lem}

\begin{proof}

    (a) Recall that  $\Phi_{i,j}(x)$, is the $j$-th coordinate of the monotonically decreasing sorted vector $\Trun(\ip{U_gw_i}{x})_{g \in G}$.
    Suppose that \[ \left|\{g\in G~,~\ip{U_g w_i}{x}>\Phi_{i,j}(x) \} \right| > j-1 .\] Then there are at least $j$, distinct elements of group $G$,  $(h_1,\dots h_j)$,  such that $\ip{U_{h_k}w_i}{x}>\ip{U_gw_i}{x},~\forall k \in [j]$. But this is a contradiction. Similarly, if \[ \left|\{g\in G~,~\ip{U_g w_i}{x}<\Phi_{i,j}(x) \} \right| > N-j \]
    then there exist at least $N-j+1$, distinct elements of group $G$, say $\{h_1,\dots h_j\}$ such that $\ip{U_{h_k}w_i}{x}<\ip{U_gw_i}{x}~\forall k \in [N-j+1]$, which is also a contradiction.
    
    (b) If $\Phi_{i,j}(x)=\Phi_{i,j+1}(x)$, a group element $g$ achieves $\Phi_{i,j}(x)$ if and only if, it also achieves $\Phi_{i,j+1}(x)$, therefore $L^{i,j}(x)=L^{i,j+1}(x)$. On the other hand, if $\Phi_{i,j}(x)>\Phi_{i,j+1}(x)$ then  we claim that $L^{i,j}(x)$ and $ L^{i,j+1}(x)$ are disjoint sets, because otherwise there is $g \in L^{i,j}(x) \cap  L^{i,j+1}(x)$, but then $\Phi_{i,j}(x)= \ip{U_gw_i}{x}=\Phi_{i,j+1}(x)$. 
    
Now assume that  \[ \{g\in G~,~\ip{U_g w_i}{x}>\Phi_{i,j+1}(x) \} \not= j .\]
Without loss of generality
     \[ \{g\in G~,~\ip{U_g w_i}{x}>\Phi_{i,j+1}(x) \} > j \] 
    so there exists at least $j+1$ group elements $(h_1,\dots,h_{j+1})$, such that \[\ip{U_{h_k} w_i}{x}>\Phi_{i,j+1},~\forall k \in [j+1]\] but this is a contradiction. Similarly,\[ \{g\in G~,~\ip{U_g w_i}{x}<\Phi_{i,j}(x) \} =
\cup_{k\geq j+1} L^{i,k}(x) 
~~,~~ \left| \cup_{k\geq j+1} L^{i,k}(x)\right| = N- j. \]
\end{proof}

Note that for any $w_1,\ldots,w_p\in\Vc\setminus\{0\}$ the subset $L^{i,j}(x) \subset G$ has the following ``nestedness'' property.
\begin{lem}\label{lemnest}
  For any $x,y \in \Vc$ such that $\norm{y} < \frac{1}{2}\Delta^{i,j}(x)$, we have that $L^{i,j}(x+y) \subset L^{i,j}(x)$. Furthermore,
\[
\{ g\in G~,~\ip{U_g w_i}{x}>\Phi_{i,j}(x) \} \subset 
\{ g\in G~,~\ip{U_g w_i}{x+y}>\Phi_{i,j}(x+y) \},
\]
\[
\{ g\in G~,~\ip{U_g w_i}{x} < \Phi_{i,j}(x) \} \subset 
\{ g\in G~,~\ip{U_g w_i}{x+y} < \Phi_{i,j}(x+y) \},
\]
\[
\{ g\in G~,~\ip{U_g w_i}{x+y}\geq\Phi_{i,j}(x+y) \} \subset 
\{ g\in G~,~\ip{U_g w_i}{x} \geq \Phi_{i,j}(x) \}
\]
and
\[
\{ g\in G~,~\ip{U_g w_i}{x+y} \leq \Phi_{i,j}(x+y) \} \subset 
\{ g\in G~,~\ip{U_g w_i}{x} \leq \Phi_{i,j}(x) \}.
\]

\end{lem}

\begin{proof}

   Suppose that exists $g \in G$ such that $g \in L^{i,j}(x+y)$ but $g \notin L^{i,j}(x)$. Without loss of generality assume that $\ip{U_gw_i}{x} < \Phi_{i,j}(x)$. Then for every $h  \in \cup_{k \le j}L^{i,k}(x)$
        \begin{align*}
      \ip{U_hw_i}{x+y}-\ip{U_gw_i}{x+y} \geq  \ip{U_hw_i}{x}-\ip{U_gw_i}{x}-2\norm{y}\norm{w_i}>0.
        \end{align*}
        On the other hand, $\ip{U_gw_i}{x+y} = \Phi_{i,j}(x+y)$. Thus 
$$\cup_{k\leq j}L^{i,k}(x)\subset \{h\in G~,~\ip{U_h w_i}{x+y}>\Phi_{i,j}(x+y)\}.$$
    But the set  $\cup_{k \le j}L^{i,k}(x)$ contains at least $j$ elements (since each $L^{i,j}(x)$ is non-empty) and so we derived a contradiction with \Cref{lemPhi}(a)  \Cref{eq:ineq1}.


\end{proof}
\begin{lem}\label{lemnest2}
For $i \in [p]$ and $j \in [N]$, fix vectors $x,y \in \Vc$ and positive numbers $c_1,c_2>0$ such that $\max(c_1,c_2)\norm{y}< \frac{1}{4}\Delta^{i,j}(x)$.
Then $L^{i,j}(x+c_1y)=L^{i,j}(x+c_2y)$.
\end{lem}

\begin{proof}
 \item Assume that exist $g_1 \in L^{i,j}(x+c_2y)$ with $g_1 \notin L^{i,j}(x+c_1y)$. 
Without loss of generality assume that $\ip{U_{g_1}w_i}{x+c_1y} < \Phi_{i,j}(x+c_1y)$. Let $q>j$ be the smallest integer such that $g_1\in L^{i,q}(x+c_1y)$. Then $\Phi_{i,q}(x+c_1y)=\ip{U_{g_1}w_i}{x+c_1y}<\Phi_{i,j}(x+c_1y)$. By \Cref{lemPhi} (b)(ii), 
\[ \left| \cup_{r\leq j} L^{i,r}(x+c_1y) \right|=q-1\geq j, \]
and $g_1\not\in\cup_{r\leq j}L^{i,r}(x+c_1y)$. 
On the other hand, from \Cref{lemPhi} (a), \Cref{eq:ineq1}, 
\[
\left| \{ h\in G ~,~\ip{U_{h}w_i}{x+c_2y}>\Phi_{i,j}(x+c_2y)\} \right|\leq j-1
\]
Hence
\[
\cup_{r\leq j} L^{i,r}(x+c_1y) \setminus 
\{ h\in G ~,~\ip{U_{h}w_i}{x+c_2y}>\Phi_{i,j}(x+c_2y)\}
\neq \emptyset
\]
Therefore there exists $h \in \cup_{r\leq j} L^{i,r}(x+c_1y)$ such that  
\begin{equation}
\label{eq:qq1}
\ip{U_{h}w_i}{x+c_2y}\leq \Phi_{i,j}(x+c_2y)=
\ip{U_{g_1}w_i}{x+c_2y}.
\end{equation}

On the other hand, by \Cref{lemnest}, $g_1 \in L^{i,j}(x)$.

But if $\ip{U_{h}w_i}{x}- \ip{U_{g_1}w_i}{x}>0$ then
\[\ip{U_{h}w_i}{x+c_2y}- \ip{U_{g_1}w_i}{x+c_2y}\geq \norm{w_i}(\Delta^{i,j}(x)- 2c_2\norm{y})>0 \]
which is a contradiction with (\ref{eq:qq1}). 
If  $\ip{U_{h}w_i}{x}- \ip{U_{g_1}w_i}{x}<0$ then
\[\ip{U_{g_1}w_i}{x+c_1y}-\ip{U_{h}w_i}{x+c_1y} \geq \norm{w_i}(\Delta^{i,j}(x)- 2c_1\norm{y})>0 \]
which is a contradiction with $h\in\cup_{r\leq j}L^{i,r}(x+c_1y)$. 
 Therefore $\ip{U_{h}w_i}{x}= \ip{U_{g_1}w_i}{x}$ and thus 
 $h \in L^{i,j}(x)$.
But then
\begin{align*}
 0 \geq \ip{U_{h}w_i}{x+c_2y}- \ip{U_{g_1}w_i}{x+c_2y} &=
 \ip{U_{h}w_i}{c_2y}- \ip{U_{g_1}w_i}{c_2y}\\&
=c_2(\ip{U_{h}w_i}{y}- \ip{U_{g_1}w_i}{y})\\
\end{align*}
and
\begin{align*}
0 <\ip{U_{h}w_i}{x+c_1y}- \ip{U_{g_1}w_i}{x+c_1y} &=
 \ip{U_{h}w_i}{c_1y}- \ip{U_{g_1}w_i}{c_1y}\\ &=
c_1(\ip{U_{h}w_i}{y}- \ip{U_{g_1}w_i}{y}).\\
\end{align*}
Which contradict each other since $c_1,c_2>0$.
   
\end{proof}

\begin{lem} \label{lemdelta}
For any $w_1,\ldots,w_p\in\Vc\setminus\{0\}$ and $x\in\Vc$, the sets $L^{i,j}(x)$ and perturbation bounds $\Delta^{i,j}(x)$ have the following properties:
\begin{enumerate}
\item For any $t>0$, $L^{i,j}(tx)=L^{i,j}(x)$. 
\item  For any $i\in[p]$, $j\in [N]$, and $t>0$, $\Delta^{i,j}(tx)=t\Delta^{i,j}(x)$. 
\item For any $i\in[p]$, $j\in [N]$, and $x\in\Vc\setminus\{0\}$, $\Delta^{i,j}(x)>0$.
\end{enumerate} 
\end{lem}

\begin{proof}

1.,2. For $t>0$, $\Phi_{i,j}(tx)=t\Phi_{i,j}(x)$, from where the claims follow from the definitions of $L^{i,j}(x)$ and $\Delta^{i,j}(x)$.

3. This claim follows from definitions of $\Delta_{i,j}$ which is the minimum of a finite set of positive numbers.

\end{proof}

\begin{lem}\label{lemhelplow}
Fix $w_i\in\Vc\setminus\{0\}$ and $j \in [N]$.
For any $k>1$, fix $z_1\in\Vc$ of unit norm, $\norm{z_1}=1$, and choose arbitrary $z_2,...,z_k$ that satisfy only the norm conditions 
\[\norm{z_{l+1}} 
\le  \min(\frac{1}{4}\Delta^{i,j}(\sum_{r=1}^l z_r),\frac{1}{4}\norm{z_l})~\forall l \in [k-1].\] 

For any scalars $a_1,\dots,a_k \in \left(1-\frac{1}{16k} \Delta^{i,j}(\sum_{r=1}^kz_r),1+\frac{1}{16k} \Delta^{i,j}(\sum_{r=1}^k z_r)\right)$ the following hold true: 
\begin{enumerate}
\item If $g_1,g_2 \in L^{i,j}(\sum_{r=1}^k z_r)$ then
\[ \ip{U_{g_1}w_i}{ \sum_{r=1}^k a_rz_r}
 =
 \ip{U_{g_2}w_i}{ \sum_{r=1}^k a_rz_r}.\]

 \item 
\begin{equation}
L^{i,j}(\sum_{r=1}^k a_r z_r) = L^{i,j}(\sum_{r=1}^k z_r),
\end{equation}
\item    
\begin{equation}
\label{eq:Deltas}
\frac{1}{4}\Delta^{i,j}(\sum_{r=1}^k  a_r z_r) <  \Delta^{i,j}(\sum_{r=1}^k  z_r) < 4\Delta^{i,j}(\sum_{r=1}^k a_r z_r).
\end{equation}
\item 

\begin{equation}
   \bigcup_{l\leq j}L^{i,l}(\sum_{r=1}^k a_r z_r) = \bigcup_{l\leq j}L^{i,l}(\sum_{r=1}^k z_r),
   \end{equation}
\begin{equation}   
\bigcup_{l\geq j}L^{i,l}(\sum_{r=1}^k a_r z_r) = \bigcup_{l\geq j}L^{i,l}(\sum_{r=1}^k z_r).
\end{equation}
\item For every $e \in \Vc$ with $\norm{e}< \frac{1}{16}\Delta^{i,j}(\sum_{r=1}^k z_r)$,
\begin{equation}
\label{Eq:lije}
L^{i,j}(\sum_{r=1}^k a_r z_r+e) = L^{i,j}(\sum_{r=1}^k z_r+e)
\end{equation}
\end{enumerate}

\end{lem}
\begin{rem}    
Notice this Lemma allows us to choose arbitrary directions for vectors $z_2,...,z_k$. Conclusions of the lemma hold true even for $z_2=\ldots=z_k=0$. On the other hand, the norm conditions $\norm{z_{l+1}}\leq \frac{1}{4}\norm{z_l}$ and $\norm{z_1}=1$ prevent $\sum_{l=1}^k z_l$ from ever reaching 0. Hence $\Delta({i,j}(\sum_{r=1}^l z_r) >0$ for all $l\in[k]$.
\end{rem}

 \begin{proof}
 \mbox{}

{\bf 1}.
Note that by norm conditions on $z_r$, \Cref{lemnest} and \Cref{lemnest2}, if $g_1,g_2 \in  L^{i,j}(\sum_{r=1}^k z_r)$ then  $ \ip{U_{g_1}w_i}{\sum_{r=1}^l z_r}= \ip{U_{g_2}w_i}{\sum_{r=1}^l z_r}~\forall l \in [k]$. Starting with $l=1$ and proceeding recursively we get $ \ip{U_{g_1}w_i}{ z_r}= \ip{U_{g_2}w_i}{ z_r}~\forall r \in [k]$. Therefore $\ip{U_{g_1}w_i}{a_rz_r}= \ip{U_{g_2}w_i}{a_rz_r}~\forall r \in [k]$ and so $\ip{U_{g_1}w_i}{ \sum_{r=1}^k a_rz_r}=\ip{U_{g_2}w_i}{ \sum_{r=1}^k a_rz_r}.$

{\bf 2}. 

"$\subset$".  
From \Cref{lemnest} we have that
$ L^{i,j}(\sum_{r=1}^k a_r z_r)\subset  L^{i,j}(\sum_{r=1}^k z_r)$.

"$\supset$".  
Let $g_1 \in L^{i,j}(\sum_{r=1}^k z_r)$.
Take any $g_2 \in L^{i,j}(\sum_{r=1}^k a_rz_r)$. By \Cref{lemnest}, $g_2 \in L^{i,j}(\sum_{r=1}^k z_r)$, but from part 1 of this Lemma,
 \[\ip{U_{g_1}w_i}{ \sum_{r=1}^k a_rz_r}=\ip{U_{g_2}w_i}{ \sum_{r=1}^k a_rz_r}.\]
 Therefore $g_1\in L^{i,j}(\sum_{r=1}^k a_r z_r)$, which proves the other inclusion.

{\bf 3}.

Case 1. Suppose $L^{i,j}(\sum_{r=1}^k z_r)=G$. 
From \cref{lemnest} we know that 
\[G=L^{i,j}(\sum_{r=1}^k z_r) \subset  L^{i,j}(\sum_{r=1}^{k-1} z_r) \subset \dots \subset L^{i,j}( z_1).\]
Τherefore, $L^{i,j}(z_r)=G,~\forall r \in [k]$ and consequently $L^{i,j}( a_r z_r)=G,~\forall r \in [k]$.

Moreover, $a_1,\dots a_k \in (7/8,9/8)$.
Therefore,
\begin{align*}  
\Delta^{i,j}(\sum_{r=1}^{k} a_rz_r)= \frac{1}{\norm{w_i}}\norm{\sum_{r=1}^{k} a_rz_r}  \le \frac{9}{8\norm{w_i}}\left(\sum_{r=1}^{k}\norm{z_r}\right)  <  \frac{3}{2\norm{w_i}}\norm{z_1}
\\ 
\end{align*}
and
\begin{align*}  
\Delta^{i,j}(\sum_{r=1}^{k} a_rz_r)= \frac{1}{\norm{w_i}}\norm{\sum_{r=1}^{k} a_rz_r}  \geq  \frac{1}{\norm{w_i}} \left(\frac{7}{8}\norm{z_1}-\frac{9}{8}\sum_{r=2}^{k}\norm{z_r}\right)  > \frac{1}{2\norm{w_i}} \norm{z_1}.
\\ 
\end{align*}
Similarly, 

\begin{align*}  
\Delta^{i,j}(\sum_{r=1}^{k} z_r)= \frac{1}{\norm{w_i}}\norm{\sum_{r=1}^{k} z_r}  \le  \frac{1}{\norm{w_i}}\left(\sum_{r=1}^{k}\norm{z_r}\right)   <  \frac{4}{3\norm{w_i}}\norm{z_1}
\\ 
\end{align*}
and
\begin{align*}  
\Delta^{i,j}(\sum_{r=1}^{k} z_r)= \frac{1}{\norm{w_i}}\norm{\sum_{r=1}^{k} z_r}  \geq \frac{1}{\norm{w_i}} \left( \norm{z_1}-\sum_{r=2}^{k}\norm{z_r}\right) > \frac{2}{3\norm{w_i}} \norm{z_1}.
\\ 
\end{align*}

So,
\[ 
\frac{1}{4}\Delta^{i,j}(\sum_{r=1}^{k} z_r) \le \Delta^{i,j}(\sum_{r=1}^{k} a_rz_r) \le 4\Delta^{i,j}(\sum_{r=1}^{k} z_r).
\]

Case 2.
Now assume that $L^{i,j}(\sum_{r=1}^k z_r) \not=G$.
Fix $g_1 \in G$ that achieves
$\Delta^{i,j}(\sum_{r=1}^k a_rz_r)$, i.e.
\[\frac{1}{\norm{w_i}}|\ip{U_{g_1} w_i}{\sum_{r=1}^k a_rz_r}-\Phi_{i,j}(\sum_{r=1}^k a_rz_r)|= \Delta^{i,j}(\sum_{r=1}^k a_rz_r)\]
and $g_2 \in L^{i,j}(\sum_{r=1}^k a_rz_r)$.

Then

\begin{align*}
   \Delta^{i,j}(\sum_{r=1}^k a_rz_r) = & \frac{1}{\norm{w_i}}\left|\ip{U_{g_1} w_i}{\sum_{r=1}^k a_rz_r}-\ip{U_{g_2} w_i}{\sum_{r=1}^k a_rz_r}\right|\\ \geq  & \frac{1}{\norm{w_i}}\left|\ip{U_{g_1}w_i}{\sum_{r=1}^k z_r}-\ip{U_{g_2}w_i}{\sum_{r=1}^k z_r}\right| -\sum_{r=1}^k 2 |1-a_r| \norm{z_r}\\ \geq &\Delta^{i,j}(\sum_{r=1}^k z_r)-2\sum_{r=1}^k  |1-a_r| \norm{z_r}>\frac{1}{2}\Delta^{i,j}(\sum_{r=1}^k z_r).
  \end{align*}

and,
\begin{align*}
   \Delta^{i,j}(\sum_{r=1}^k a_rz_r) = & \frac{1}{\norm{w_i}}|\ip{U_{g_1} w_i}{\sum_{r=1}^k a_rz_r}-\ip{U_{g_2} w_i}{\sum_{r=1}^k a_rz_r}|\\ \le  & \frac{1}{\norm{w_i}}|\ip{U_{g_1}w_i}{\sum_{r=1}^k z_r}-\ip{U_{g_2}w_i}{\sum_{r=1}^k z_r}| +\sum_{r=1}^k 2 |1-a_r| \norm{z_r}\\ \geq &\Delta^{i,j}(\sum_{r=1}^k z_r)+2\sum_{r=1}^k  |1-a_r| \norm{z_r}\le 2\Delta^{i,j}(\sum_{r=1}^k z_r).\\
  \end{align*}

 Therefore,

  \[\frac{1}{2}\Delta^{i,j}(\sum_{r=1}^k a_rz_r) \le \Delta^{i,j}(\sum_{r=1}^k a_rz_r) \le 2\Delta^{i,j}(\sum_{r=1}^k a_rz_r).\]

{\bf 4}.

Let $g_1 \in   \cup_{l=1}^jL^{i,l}(\sum_{r=1}^kz_r)$.
If $g_1 \in L^{i,j}(\sum_{r=1}^kz_r)$ then we just showed that $g_1 \in L^{i,j}(\sum_{r=1}^k a_rz_r) $.
If $g_1 \notin L^{i,j}(\sum_{r=1}^kz_r)$ then  for every $h \in \cup_{l\geq j}L^{i,l}L^{i,l}(\sum_{r=1}^k z_r)$
\begin{align*}
\ip{U_{g_1} w_i}{\sum_{r=1}^k a_r z_r}-\ip{U_{h} w_i}{\sum_{r=1}^k a_r z_r} &\geq \ip{U_{g_1} w_i}{\sum_{r=1}^k z_r}-\norm{w_i}\norm{\sum_{r=1}^k|1-a_r|z_r}\\ &>\norm{w_i}( \Delta^{i,j}(\sum_{r=1}^k z_r)-\norm{\sum_{r=1}^k|1-a_r|z_r}) \\ &>0.
\end{align*}
But from \Cref{lemPhi} $|\cup_{l\geq j}L^{i,l}L^{i,l}(\sum_{r=1}^k z_r)|\geq N-j+1$. So $g_1 \in \cup_{l=1}^jL^{i,l}(\sum_{r=1}^kz_r)$.
Therefore,
\[  \cup_{l\le j}L^{i,l}(\sum_{r=1}^kz_r) \subset \cup_{l \le j}L^{i,l}(\sum_{r=1}^ka_rz_r).\]

The other inclusions are obtained similarly.

{\bf 5}.
We prove the equality between complements: 
\[\left(L^{i,j}(\sum_{r=1}^k z_r +e)\right)^c = \left(L^{i,j}(\sum_{r=1}^k a_rz_r+e)\right)^c.\]

First notice that by \Cref{lemnest} we have that $L^{i,j}(\sum_{r=1}^k z_r+e) \subset L^{i,j}(\sum_{r=1}^k z_r)$ and $L^{i,j}(\sum_{r=1}^k a_rz_r+e) \subset L^{i,j}(\sum_{r=1}^k a_rz_r)$.
    From \Cref{lemhelplow} part (2) we have that 
$L^{i,j}(\sum_{r=1}^k z_r)= L^{i,j}(\sum_{r=1}^k a_rz_r)$.
Take $g\in L^{i,j}(\sum_{r=1}^k z_r+e)$ and $h\in\left(L^{i,j}(\sum_{r=1}^k z_r+e)\right)^c$.
Hence $\ip{U_g w_i}{\sum_{r=1}^k z_r+e}\neq \ip{U_h w_i}{\sum_{r=1}^k z_r+e}$.

 There are two cases: 

Case 1. $h\in L^{i,j}(\sum_{r=1}^k z_r)\setminus L^{i,j}(\sum_{r=1}^k z_r+e)$. Thus $\ip{U_g w_i}{\sum_{r=1}^k z_r}=\ip{U_h w_i}{\sum_{r=1}^k z_r}$.
Therefore $\ip{U_g w_i}{e}\neq\ip{U_h w_i}{e}$. On the other hand $h\in L^{i,j}(\sum_{r=1}^k a_r z_r)$ since 
 $L^{i,j}(\sum_{r=1}^k z_r)= L^{i,j}(\sum_{r=1}^k a_rz_r)$. 
 Hence $\ip{U_g w_i}{\sum_{r=1}^k a_r z_r}=\ip{U_h w_i}{\sum_{r=1}^k a_r z_r}$, which implies 
$\ip{U_g w_i}{\sum_{r=1}^k a_rz_r+e} \neq \ip{U_h w_i}{\sum_{r=1}^k a_rz_r+e}$. Thus $h\in \left(L^{i,j}(\sum_{r=1}^k a_rz_r+e)\right)^c$.

Case 2. $h\in G\setminus L^{i,j}(\sum_{r=1}^k z_r)$.  Thus
$\ip{U_g w_i}{\sum_{r=1}^k z_r}\neq\ip{U_h w_i}{\sum_{r=1}^k z_r}$. In this case
$|\ip{U_h w_i}{\sum_{r=1}^k z_r} - \ip{U_g w_i}{\sum_{r=1}^k z_r}|\geq \norm{w_i}\Delta^{i,j}(\sum_{r=1}^k z_r)$ and 
\[ \left|\ip{U_h w_i}{\sum_{r=1}^k a_rz_r+e} - \ip{U_g w_i}{\sum_{r=1}^k a_rz_r+e}\right| \geq \]

\begin{align*}
     \geq
\left| \ip{U_h w_i}{\sum_{r=1}^k z_r} - \ip{U_g w_i}{\sum_{r=1}^k z_r} \right|&-
\left| 
\ip{U_h w_i}{\sum_{r=1}^k (a_r-1)z_r+e}
\right|\\ &-
\left| 
\ip{U_g w_i}{\sum_{r=1}^k (a_r-1)z_r+e}
\right|
\geq 
\end{align*}
\[
\geq \norm{w_i}\left[ \Delta^{i,j}(\sum_{r=1}^k z_r) - 2 \left( \sum_{r=1}^k |a_r-1| \norm{z_r} + \norm{e} \right)\right]>0\hspace{20mm}(*)\]
Hence again $h\in \left(L^{i,j}(\sum_{r=1}^k a_rz_r+e)\right)^c$.
This proves that $\left(L^{i,j}(\sum_{r=1}^k z_r+e)\right)^c\subset \left(L^{i,j}(\sum_{r=1}^k a_r z_r+e)\right)^c$. 
The reverse inclusion is shown similarly, with $\Delta^{i,j}(\sum_{r=1}^kz_r)$ replaced by $\Delta^{i,j}(\sum_{r=1}^k a_rz_r)$ in (*). 
 \end{proof}

\subsubsection{Positivity of the Lower Lipschitz Constant}

Now we prove that the lower Lipschitz bound must be positive if the embedding map $ \hat{\Phi}_{\bw,S}$ is injective. We do so by contradiction. 

The strategy is the following:
 Assume the lower Lipschitz constant is zero. 
\begin{itemize}
    
\item First we find a unit norm vector $z_1$ where the local lower Lipschitz constant vanishes.

\item Next we construct inductively a sequence of non-zero vectors $z_2,z_3,...,z_k$ so that the local lower Lipschitz constant vanishes in a convex set of
the form $\{\sum_{r=1}^k a_r z_r~,~|a_r-1|<\delta\}$
 for some $\delta>0$ small enough, and where sets $L^{i,j}$ remain constant. These steps are depicted in Figure \ref{robarm} that suggests a "robotic arm" procedure (this name was suggested by the authors of \cite{qaddura2024stable}).

\item For $k=d$ this construction defines a non-empty open set where the local lower Lipschitz constant vanishes and $L^{i,j}$ remain constants. This allows us to construct $u,v\neq 0$ so that $x=u+\sum_{r=1}^d z_r$ and $y=v+\sum_{r=1}^d z_r$ satisfy $x\not\sim y$ and yet $\Phi_{\bw,S}(x)=\Phi_{\bw,S}(y)$. This contradicts the injectivity hypothesis.
\end{itemize}

First, we show that if the lower bound is zero then it can be achieved locally. 
With a slightly abuse of notation we define 

\begin{lem}\label{lemassum}
    Fix $\bw=(w_1,\dots,w_p)\in \Vc^p$ and $S\subset [N]\times [p]$.
    If the lower Lipschitz constant of map $\Phi_{\bw,S}$ is zero, then there exist sequences $(x_{n})_{n}$, $(y_{n})_{n}$ in $\Vc$ such that
    \[\lim_{n \to \infty} \frac{\norm{\Phi_{\bw,S}(x_{n})-\Phi_{\bw,S}(y_{n})}^2}{d(x_{n},y_{n})^2}=0\]
 and, additionally, satisfy the following relations:
\begin{enumerate}
 \item (convergence) They share a common limit $z_1$,  
 \begin{equation} 
 \label{eq17}
 \lim_{n\rightarrow\infty}x_{n} = \lim_{n\rightarrow\infty} y_{n}=z_1, 
 \end{equation} 
 with $\norm{z_1}=1$;
 
 \item (boundedness) For all $k$:
    \begin{gather}
    \label{eq18}
        \norm{x_{n}}=1 \\ 
        \label{eq19}
        \norm{y_{n}} \le 1
    \end{gather}

    \item (alignment) For all $k$:
    \begin{gather}
    \label{eq20}
        \norm{x_{n}-y_{n}}= \min_{g \in G}\norm{x_{n}-U_gy_{n}}\\
     \label{eq21}
        \norm{x_{n}-z_1}= \min_{g \in G}\norm{x_{n}-U_gz_1} \\
     \label{eq22}
        \norm{y_{n}-z_1}= \min_{g \in G}\norm{y_{n}-U_gz_1}    \end{gather}
\end{enumerate}
\end{lem}
\begin{proof}
Because the lower Lipschitz bound of map $ \Phi_{\bw,S}$ is zero we have that
\[\inf_{\substack{x,y \in \Vc \\ x \nsim y}} \frac{\norm{\Phi_{\bw,S}(x)-\Phi_{\bw,S}(y)}^2}{\dis([x],[y])^2}=0.\]
Thus, we can find sequences $(x_n)_n,(y_n)_n \in \Vc$  such that
\[\displaystyle \lim_{n \to \infty} \frac{\norm{\Phi_{\bw,S}(x_n)-\Phi_{\bw,S}(y_n)}^2}{\dis([x_n],[y_n])^2}=0\]

Now, notice that for all $t>0$ we have $\Phi_{\bw,S}(tx) = t\Phi_{\bw,S}(x)$ and $\dis([tx],[ty])=t \dis([x],[y])$. So, for every $t>0$
\begin{align*}
	\frac{\norm{\Phi_{\bw,S}(x_n)-\Phi_{\bw,S}(y_n)}^2}{\dis([x_n],[y_n])^2} 
    \\ &=\frac{\norm{\Phi_{\bw,S}(tx_n)-\Phi_{\bw,S}(ty_n)}^2}{\dis([tx_n],[ty_n])^2}.
\end{align*}
By setting $t = \frac{1}{\max(\norm{x_n}, \norm{y_n})}$ we can always assume that both $x_n$ and $y_n$, lie in the unit ball, and what is more thanks to the symmetry of the formulas we can additionally assume that one of the sequences, say $x_n$, lies on unit sphere. In other words, $\norm{x_n}=1$ and $\norm{y_n}\le 1$ for all $n\in \N$.

Because of this, we can find a convergent subsequence $(x_{n_k})_k$ of $(x_{n})_n$ with $x_{n_k} \to x_{\infty}$. Similarly, we can find a convergent subsequence $(y_{n_{k_l}})_l$ of $(y_{n_k})_n$ with $y_{n_{k_l}} \to y_{\infty}$. Clearly, $x_{n_{k_l}} \to x_{\infty}$. For easiness of notation, we denote the sequences $(x_{n_{k_l}})_l$ and $(y_{n_{k_l}})_l$ by $(x_{n})_n$ and $(y_{n})_n$, respectively.

Next, suppose that $x_{\infty} \nsim y_{\infty}$. Then,
\[\frac{\norm{\Phi_{\bw,S}(x_{\infty})-\Phi_{\bw,S}(y_{\infty})}^2}{\dis([x_{\infty}],[y_{\infty}])^2}=\displaystyle \lim_{k \to \infty} \frac{\norm{\Phi_{\bw,S}(x_{n})-\Phi_{\bw,S}(y_{n})}^2}{\dis([x_{n}],[y_{n}])^2}=0,\]
and thus, $\Phi_{\bw,S}(x_{\infty})= \Phi_{\bw,S}(y_{\infty})$, which contradict the injectivity assumption. Hence, $x_{\infty} \sim y_{\infty}$.

Now, let us denote by $g_{\infty}$ a group element such that $x_{\infty}= U_{g_{\infty}}y_{\infty}$.
Observe that $\lim_{n \to \infty} \norm{x_{n}-U_{g_{\infty}}y_n} =0$.
For each $n \in \N$ there exists at least one element $g_n \in G$, which achieves the Euclidean distance between $x_n$ and $U_{g_{\infty}}y_n$, i.e. satisfying $\dis([x_n],[U_{g_\infty} y_n]) = \norm{x_n- U_{g_kg_{\infty}}y_n}$.
But $G$ is a finite group, meaning that, as $n$ goes to infinity, there must exist an element $g_0\in G$ for which $g_n=g_0$ for infinitely many $n$. Let $(n_m)_m$ be the sequence of all such indices. We see that $\dis([x_{n_m}],[U_{g_{\infty}}y_{n_m}])=\norm{x_{n_m}- U_{g_0g_{\infty}}y_{n_m}}$ for all $m\in \N$. Finally, for every $m \in \N$, let $g_m \in G$ be a group element that achieves the Euclidean distance between $x_{n_m}$ and $x_{\infty}$, that is
\[\dis([x_{n_m}],[x_{\infty}])=\norm{U_{g_m}x_{n_m}-x_{\infty}}.\]
Denote $U_{g_{m}}x_{n_m}$ by $x_{n}$ and $U_{g_{m}g_0g_{\infty}}y_{n_m}$ by $y_{n}$. 
So far we obtained two sequences $(x_n)_n$ and $(y_n)_n$ that satisfy (\ref{eq17}-\ref{eq21}).
Now let $h_n\in G$ denote a group element so that $d(y_n,z_1)=\norm{y_n-U_{h_n}z_1}$. Since $G$ is finite, pass to a subsequence (again indexed by $n$) so that $h_n=h_0$. Therefore $d(y_n,z_1)=\norm{y_n-U_{h_0}z_1}\leq \norm{y_k-z_1}$. But $\lim_{n \to \infty}y_n=z_1$. 
Thus $U_{h_0}z_1=z_1$.  This shows (\ref{eq22}) and the lemma is now proved.
\end{proof}

\begin{figure} 
    \centering
\begin{tikzpicture}[tdplot_main_coords, scale = 0.7]

\coordinate (P) at ({5},{0},{5});

\shade[ball color = lightgray, opacity = 0.5] (P) circle (1cm);
\draw (P) node[above] {$z_1$};

\fill (P) circle (2pt);

\draw[-stealth] (0,0,0) -- (6,0,0) node[anchor=north east]{$x$};
\draw[-stealth] (0,0,0) -- (0,6,0) node[anchor=north west]{$y$};
\draw[-stealth] (0,0,0) -- (0,0,6) node[anchor=south]{$z$};
\draw[dashed, gray] (0,0,0) -- (-1,0,0);
\draw[dashed, gray] (0,0,0) -- (0,-1,0);

\draw[thick, -stealth, nodes={opacity=0.5}] (0,0,0) -- (P);

\draw[dotted, thick, blue] ({5 + 2.5*cos(45)}, {0 + 2.5*sin(45)}, {5}) .. controls ({5 + 1.5*cos(10)}, {0 + 1.5*sin(10)}, {5}) and ({5 + 0.2}, {0.05}, {5}) .. (P);
\node[blue] at ({5 + 2.6*cos(45)}, {0 + 2.6*sin(45) - 0.8}, {5}) {$x_n$};

\draw[dotted, thick, red] ({5 + 2.5*cos(315)}, {0 + 2.5*sin(315)}, {5}) .. controls ({5 + 1.5*cos(350)}, {0 + 1.5*sin(350)}, {5}) and ({5 + 0.2}, {-0.05}, {5}) .. (P);
\node[red] at ({5 + 2.6*cos(315)}, {0 + 2.6*sin(315)}, {5}) {$y_n$};

\draw[fill = lightgray!50] (P) circle (0.5pt);

\end{tikzpicture}
    \caption{Local analysis of sequences $x_n$ and $y_n$ converging to $z_1$.}
\end{figure}

In what follows, we will denote by $H(z)$ the stabilizer group of $z$; recall that
\[H(z)=\{g \in G : U_gz=z\}.\]
 For a fixed vector $z$ we define the strictly positive number \[\rho_0(z)=\begin{cases}
     \min_{g \in G \setminus H(z)} \norm{z -U_gz},~\text{if}~H(z) \neq G\\
     \norm{z},~\text{if}~H(z) = G.
 \end{cases}\] 
Assume $N_0$ is large enough so that $d(x_{1,k},z_1)<\frac{1}{8}\rho_0(z_1)$ and $d(x_{1,k},y_{1,k})<\frac{1}{8}\rho_0(z_1)$ for all $k>N_0$. Then \[\norm{y_{1,k}-z_1} \leq \norm{y_{1,k}-x_{1,k}}+\norm{x_{1,k}-z_1}=d(y_{1,k},x_{1,k})+d(x_{1,k},z_1)<\frac{\rho_0(z_1)}{4}.\] 
 
\begin{lem}\label{lemalign}
    Assume that $\norm{u},\norm{v} < \frac{1}{4} \Delta_{0}(z_1)$ and let $x=z_1+u$ and $y=z_1+v$. Then, the following properties hold:
    \begin{enumerate}
        \item $\dis([x],[z_1])=\norm{u}$ and $\dis([y],[z_1])=\norm{v}$,
        \item $\dis([x],[y]) = \min_{g \in H(z_1)} \norm{u- U_gv} = \min_{g \in H(z_1)} \norm{U_gu-v}$, and
        \item the following are equivalent:
        \begin{enumerate}
            \item $\dis([x],[y])= \norm{u-v}$,
            \item $\norm{u-v} \le \norm{U_gu-v}$, for all $g \in H(z_1)$,
            \item $\ip{u}{v} \geq  \ip{U_gu}{v}$, for all $g \in H(z_1)$.
        \end{enumerate}
    \end{enumerate}
\end{lem}
\begin{proof}
    \begin{enumerate}
        \item If $u=0$ then the claim follows.  If $u\not=0$, then
        $\dis([x],[z_]1) = \min_{g \in G}\norm{x-U_gz_1}= \min_{g \in G}\norm{z_1-U_gz_1+u} \le \norm{u}$.
        From the other hand, suppose that minimum is achieved for a permutation $g \in G$.
        If $g \in H(z_1)$, then $\dis([x],[z_1])=\norm{u}$.
        If $g \notin H(z_1)$, then  $\dis([x],[z_1]) >\norm{u} \le \dis([x],[z_1])$, which is a contradiction.
        \item Obviously $\dis([x],[z_1]) \le \min_{g \in H(z_1)} \norm{U_gu-v}$.
        On the other hand, for $g \in G \setminus K$ and $h \in G$, 
        \begin{align*}
         \norm{U_gx-y} &= \norm{U_gz_1-z_1+U_gu-v}\\ &\geq \norm{U_gz_1-z_1}-\norm{u}-\norm{v} \\ &\geq \rho_0(z_1) -2\norm{u}-2\norm{v}+\norm{U_hu-v} \\ & \geq \dis([x],[y]).
         \end{align*}
         
        \item \begin{itemize}
            \item $(a) \Rightarrow (b)$.
            If $\dis([x],[y]) =\norm{u-v}$, then
            $\norm{u-v} \le \norm{U_gx-y}= \norm{U_gz_1-z_1+U_gu-v}$, $\forall g \in G$.
            For $g \in H(z_1)$ this reduces to (b)
            \item $(b) \Rightarrow (a)$.
            Assume that $\forall g \in H(z_1)$, $\norm{u-v} \le \norm{U_gu-v}$
            Then $\norm{u-v} = \norm{x-y} \le \norm{U_gu-v}= \norm{U_gx-y}$
            For, $g\in G\setminus H(z_1)$
            $ \norm{U_gx-y} =  \norm{U_gz_1-z_1+U_gu-v} \geq \norm{U_gz_1-z_1}-\norm{u}-\norm{v} \geq \rho_0(z_1) -2\norm{u}-2\norm{v}+\norm{u-v} \geq \norm{u-v}=\norm{x-y} $
            Thus, $\dis([x],[y])=\norm{x-y}=\norm{u-v}$
            \item $(b) \Leftrightarrow (c)$ is immediate from definition of inner product
        \end{itemize}
    \end{enumerate}
\end{proof}
\begin{rem}    
Applying Lemma \ref{lemalign} to two sequences $(x_k)_k$ and $(y_k)_k$ that satisfy (\ref{eq17}-\ref{eq20}) in \Cref{lemassum}, it follows that $d(x_{k},z_1)=\norm{x_{k}-z_1}$ and $d(y_{k},z_1)=\norm{y_{k}-z_1}$ for $k$ large enough.
 Hence alignment must occur from some rank on.
\end{rem}

\begin{lem}\label{lem4.4}
For  fixed $i \in [p]$, $ j\in S_i$ and two sequences $(x_{n})_n$, $(y_{n})_n$ produced by  \Cref{lemassum}, we denote by $g_{1,n,i,j}$  the group elements that achieves $\Phi_{i,j}(x_{n})$ and by  $g_{2,n,i,j}$ the group element that achieves the $\Phi_{i,j}^j(y_{n})$. That is $\Phi_{i,j}(x_{n})= \ip{U_{g_{1,n,i,j}}w_i}{x_{n}}$ and $\Phi_{i,j}(y_{n})= \ip{U_{g_{2,n,i,j}}w_i}{y_{n}}$. 

	We can find a sequence of natural numbers $(n_r)_r$, such that,  $g_{1,n_r,i,j} = g_{1,i,j}$ and $g_{2,n_r,i,j} = g_{2,i,j}$ $\forall r \in \N,~ i \in [p],~ j \in S_i$.
 
\end{lem} 
\begin{proof}
	For $i=1,~j=1$ there is a subsequence $(x_{n_m})_m$ such that $g_{1,1,1,n_m}= g_{1,1,1}$ for every $m \in \N$.
	Similarly, for $i=1,~j=2$ we can find a subsequence of $(x_{n_m})_m$, lets call it $(x_{n_{l}})_l$, such that $g_{1,1,2,n_{l}}= g_{1,1,2},~\forall l \in \N$. So by induction after $\sum_{i \in [p]} m_i =m$ steps we construct a subsequence of $(x_{n})_n$ lets call it $(x_{n_m})_m$ such that  $g_{1,i,j,n_m}= g_{1,i,j}$ for every $i\in [p],j \in S_i $.
	Starting from sequence $(y_{1,n_m})_m$ we repeat the same procedure concluding in a subsequence $(y_{1,n_r})_r$ such that $g_{2,i,j,n_r}= g_{2,i,j}$ for every $r\in \N, i\in [p],j \in S_i$ .
	Notice that sequences $(x_{n_r})_r$ and $(y_{n_r})_r$ that from now on  we will call them $(x_{n})_n$ and $(y_{n})_n$ for easiness of notation, satisfy the assumptions of lemma.
\end{proof}
For sequences $(x_{n})_n$,  $(y_{n})_n$ and $z_1$ defined before, let $u_{n} = x_{n}-z_1$ and $v_{n} = y_{n}-z_1$.  
Notice that
\begin{align*}\norm{\Phi_{\bw,S}(x_{n})-\Phi_{\bw,S}(y_{n})}^2=\sum_{i=1}^p \sum_{j \in S_i} |&\ip {U_{g_{1,i,j}}w_i}{x_{n}}-\ip{U_{g_{2,i,j}}w_i}{y_{n}}|^2\\ 
=\sum_{i=1}^p\sum_{j \in S_i} |&\ip{U_{g_{1,i,j}}w_i-U_{g_{2,i,j}}w_i}{z_1}\\+& \ip{w_i}{U_{g_{1,i,j}^{-1}}u_{n}-U_{g_{2,i,j}^{-1}}v_{n}}|^2.
\end{align*}
This sequence converge to $0$, as $k \to \infty$ while also $u_{n},v_{n} \to 0$.
So we conclude that for each $i \in [p]$ and $j \in S_i$, 
$\ip{U_{g_{1,i,j}}w_i-U_{g_{2,i,j}}w_i}{z_1}=0$.
So
\[\norm{\Phi_{\bw,S}(x_{n})-\Phi_{\bw,S}(y_{n})}^2= \sum_{i=1}^p \sum_{j \in S_i} |\ip{w_i}{U_{g_{1,i,j}^{-1}}u_{n}-U_{g_{2,i,j}^{-1}}v_{n}}|^2.\]
Thus we have
\begin{equation}\label{eq0}
\lim_{n \to \infty}\frac{ \sum_{i=1}^p \sum_{j \in S_i}|\ip{w_i}{U_{g_{1,i,j}^{-1}}u_{n}-U_{g_{2,i,j}^{-1}}v_{n}}|^2}{\norm{u_{n}-v_{n}}^2}=0
\end{equation}
where $\norm{u_{n}},\norm{v_{n}} \to 0$, so for large enough $n$ we have that $\norm{u_{n}}, \norm{v_{n}}\le \frac{1}{4}\rho_0(z_1)$. Recall that from \Cref{lemalign}, we conclude that exists $N_0 \in \N$, such that $\norm{u_{n}-v_{n}}\le \norm{U_gu_{n}-v_{n}}$ for all $g \in H(z_1)$ and $k\geq N_0$.

\begin{lem}\label{lemsub}
Fix  $p \in \N$, $\bw \in \Vc^p$ and $S \subset [N]\times [p]$.
Let $\Delta: \Vc \to \R$, where $\Delta(x)= \min_{(i,j) \in [p] \times [N]}\Delta^{i,j}(x)$, where the map $\Delta^{i,j}$ is defined in \eqref{DIJ}.
Fix nonzero vectors $z_1,\dots,z_k \in \Vc$, such that
\begin{align*}
\norm{z_1}=1,~\ip{z_i}{z_j}=0,~\forall i,j \in [k],~ i \not=j\\   
\intertext{and}
\norm{z_{l+1}} \le \min\left(\frac{1}{4}\Delta(\sum_{r=1}^l z_r),\frac{1}{4}\norm{z_l}\right),~\forall l \in [k-1].
\end{align*}
 Assume that the local lower Lipschitz constant of $\Phi_{\bw,S}$ vanishes at $z_1+z_2+\cdots+z_k$.
\begin{enumerate}
\item 

The local lower Lipschitz constant vanishes on the non-empty convex box 
 $\{\sum_{r=1}^k a_r z_r~,~|a_r-1|<\frac{1}{16k} \Delta(\sum_{l=1}^k z_l)\}$ centered at $z_1+z_2+\cdots+z_k$.

\item Assume $\hat{\Phi}_{\bw,S}$ is injective. 
If $k<d$ then there exists a nonzero vector $z_{k+1}$ such that:

(i) $\ip{z_{k+1}}{z_j}=0,~\forall j \in [k]$;

(ii) $\norm{z_{k+1}} \le \min\left(\frac{1}{4}\Delta(\sum_{r=1}^k z_r),\frac{1}{4}\norm{z_k}\right)$; and 

(iii) The local lower Lipschitz constant vanishes at $z_1+z_2+\cdots+z_{k+1}$, i.e. there are sequences of vectors $(x_n)_n,(y_n)_n$ such that 
\[\lim_{n \to \infty} x_n =\lim_{n \to \infty} y_n = \sum_{r=1}^{k+1} z_r\] and \[\lim_{n \to \infty}\frac{\norm{\Phi_{\bw,S}(x_n)-\Phi_{\bw,S}(y_n)}^2}{\dis([x_n],[y_n])^2}=0.\]

\end{enumerate}

 \end{lem}

\begin{proof}
\begin{enumerate}
    \item 
Let $(x_n)_n, (y_n)_n$ be sequences in $\Vc$ such that 
\[\lim_{n\to \infty} x_n =\lim_{n\to \infty}y_n= \sum_{r=1}^k z_r\] and 
\[\lim_{n \to \infty} \frac{\norm{\Phi_{\bw,S}(x_n)-\Phi_{\bw,S}(y_n)}^2}{\dis([x_n],[y_n])^2}=0.\]

Claim: For any $a_1,\dots,a_k \in \left(1-\frac{1}{16k} \Delta(\sum_{r=1}^kz_r),1+\frac{1}{16k} \Delta(\sum_{r=1}^k z_r)\right)$ the sequences 
\[\tilde{x}_n=x_n+\sum_{r=1}^k (a_r-1)z_r\] and 
\[\tilde{y}_n=y_n+\sum_{r=1}^k (a_r-1)z_r\] also achieve a zero lower Lipschitz constant, i.e.
\[\lim_{n \to \infty}\frac{\norm{\Phi_{\bw,S}(\tilde{x}_n)-\Phi_{\bw,S}(\tilde{y}_n)}^2}{\dis([\tilde{x}_n],[\tilde{y}_n])^2}=0.\]

    First we denote by $u_n$ and $v_n$ the difference sequences $x_n$ and $y_n$ to their common limit $\sum_{r=1}^kz_r$, 
    \[u_n=x_n-\sum_{r=1}^kz_r=\tilde{x}_n-\sum_{r=1}^ka_rz_r\] and
    \[v_n=y_n-\sum_{r=1}^kz_r=\tilde{y}_n-\sum_{r=1}^ka_rz_r.\]
   Sequences $(u_n)_n$ and $(v_n)_n$ converge to zero. Therefore there exists $M_0 \in \N$ such that $\forall n \geq M_0$ 
    \begin{enumerate}
        \item $\norm{u_n}=\norm{x_n-\sum_{r=1}^kz_r} < \frac{1}{16}\Delta(\sum_{r=1}^kz_r)$
        \item  $\norm{u_n}=\norm{\tilde{x}_n-\sum_{r=1}^ka_rz_r} < \frac{1}{16}\Delta(\sum_{r=1}^ka_rz_r)$
        \item  $\norm{v_n}=\norm{y_n-\sum_{r=1}^kz_r} < \frac{1}{16}\Delta(\sum_{r=1}^kz_r)$
         \item  $\norm{v_n}=\norm{\tilde{y}_n-\sum_{r=1}^ka_rz_r} < \frac{1}{16}\Delta(\sum_{r=1}^ka_rz_r)$.
    \end{enumerate}
    
    Thus from part (3) of \Cref{lemhelplow}, \Cref{lemnest} and part (2) of \Cref{lemhelplow} we have that for any $n\geq M_0$ and $(i,j)\in S$ \[L^{i,j}(\tilde{x}_n)= L^{i,j}(x_n)\subset L^{i,j}(\sum_{r=1}^kz_r) =L^{i,j}(\sum_{r=1}^ka_rz_r)\]
    and
    \[L^{i,j}(\tilde{y}_n) =L^{i,j}(y_n)\subset L^{i,j}(\sum_{r=1}^kz_r) =L^{i,j}(\sum_{r=1}^ka_rz_r) .\] 
    Therefore,
    \begin{align*}
        0= &\lim_{n \to \infty}\frac{\norm{\Phi_{\bw,S}(x_n)-\Phi_{\bw,S}(y_n)}^2}{\dis([x_n],[y_n])^2}= \\
        =&\lim_{n \to \infty}\frac{ \sum_{i=1}^p \sum_{j \in S_i}|\ip{w_i}{U_{g_{1,i,j}^{-1}}u_n-U_{g_{2,i,j}^{-1}}v_n|^2}}{\norm{u_n-v_n^2}}= \\= & \lim_{n \to \infty}\frac{\norm{\Phi_{\bw,S}(\tilde{x}_n)-\Phi_{\bw,S}(\tilde{y}_n)}^2}{\dis([\tilde{x}_n],[\tilde{y}_n])^2},
    \end{align*}
where \[g_{1,i,j} \in L^{i,j}(x_n)=L^{i,j}(\tilde{x}_n)\] and   
\[g_{2,i,j} \in L^{i,j}(y_n)=L^{i,j}(\tilde{y}_n).\]
This proves the lower Lipschitz constant of $\Phi_{\bw,S}$ vanishes at $\sum_{r=1}^ka_rz_r$.

\item 
Let two sequences $(x_n)_n,(y_n)_n$ that both converge to $\sum_{r=1}^k  z_r$, and achieve lower Lipschitz bound zero for map $\Phi_{\bw,S}$.
We align sequences $(x_n)_n$ and  $(y_n)_n$ to satisfy the properties of \Cref{lemassum}.
 We denote by $a_n= P_{E_k}x_n$ and $b_n= P_{E_k}y_n$ the orthogonal projections of the sequences $(x_n)_n$ and $(y_n)_n$ respectively, on the linear subspace $E_k=\Span\{z_1,\dots,z_k\}^\bot$. 

Claim 1: First we will show that $\exists M_0$ such that $\forall n\geq M_0$, $a_n\neq 0$ or $b_n\neq 0$. Assuming otherwise, there are two sequences of vectors $x_n=\sum_{r=1}^k c_{r,n}z_r$ and $y_n=\sum_{r=1}^k d_{r,n}z_r$, where $\lim_{n \to \infty}c_{r,n}=\lim_{n \to \infty}d_{r,n}=1,~\forall r \in [k]$ that achieve lower Lipschitz bound zero. 
Recall that from part (2) of \Cref{lemhelplow} we have that $\exists M_0 \in \N$ such that $\forall n \geq M_0$ and $(i,j) \in S$
\[L^{i,j}(\sum_{r=1}^k c_{r,n}z_r)=L^{i,j}(\sum_{r=1}^k d_{r,n}z_r)=L^{i,j}(\sum_{r=1}^k z_r).\]
Then, for $g_{i,j} \in L^{i,j}(\sum_{r=1}^k z_r) $,

\begin{align*}
0&=\lim_{n \to \infty} \frac{\norm{\Phi_{\bw,S}(x_n) -\Phi_{\bw,S}(y_n)}^2}{\dis([x_n],[y_n])^2}\\
& =\lim_{n \to \infty} \frac{ \sum_{i=1}^p \sum_{j \in S_i}|\ip{U_{g_{i,j}} w_i}{x_n-y_n}|^2}{\dis([x_n],[y_n])^2}\\ 
&\geq \lim_{n \to \infty}\frac{ \sum_{i=1}^p \sum_{j \in S_i}|\ip{U_{g_{i,j}} w_i}{\sum_{r=1}^k(c_{r,n}-d_{r,n})z_r}|^2}{\norm{\sum_{r=1}^k(c_{r,n}-d_{r,n})z_r}^2}\\
&= \sum_{i=1}^p \sum_{j \in S_i}|\ip{U_{g_{i,j}} w_i}{\tilde{z}}|^2,
\end{align*}

where 
\[\tilde{z} =  \lim_{m \to \infty} \frac{\sum_{r=1}^k(c_{r,n_m}-d_{r,{n_m}})z_r}{\norm{\sum_{r=1}^k(c_{r,n_m}-d_{r,{n_m}})z_r}}\]
is a unit vector obtained as the limit of a convergent subsequence of the sequence of unit vectors
$\frac{\sum_{r=1}^k(c_{r,n}-d_{r,{n}})z_r}{\norm{\sum_{r=1}^k(c_{r,n}-d_{r,n})z_r}}$.
 Since the group $G$ is finite, we can find a positive number $\epsilon >0$ such that $\epsilon \norm{\tilde{z}} <\frac{1}{4}\Delta(\sum_{r=1}^kz_r)$ and 
$\sum_{r=1}^kz_r \nsim \sum_{r=1}^kz_r+\epsilon \tilde{z}$. In this case  
\[\Phi_{\bw,S}(\sum_{r=1}^kz_r)=\Phi_{\bw,S}(\sum_{r=1}^kz_r+\epsilon \tilde{z})\]
which contradict the injectivity property. This establishes Claim 1.
\vspace{5mm}

Now we can assume for all $n\geq M_0$, $a_n=P_{E_k}x_n\neq 0$ or $b_n=P_{E_k}y_n\neq 0$.
If need be, pass to a subsequence and/or switch the definitions of $x_n$ and $y_n$, so that  $\norm{b_n}\geq\norm{a_n}$ for all $n$. In doing so we no longer claim the normalization (\ref{eq18}). Nevertheless, both $\norm{x_n},\norm{y_n}\leq 1$.

Let $c_{r,n}, d_{r,n}$ be the unique coefficients determined by $x_n=\sum_{r=1}^k c_{r,n}z_r+a_n$, $y_n=\sum_{r=1}^k d_{r,n}z_r + b_n$. Note $\lim_{n\rightarrow\infty}c_{r,n}=\lim_{n\rightarrow\infty}d_{r,n}=1$. 

 Let $e_n= \sum_{r=1}^k (d_{r,n}-c_{r,n})z_r+b_n$ and \[s_n= \frac{\min\left(\norm{z_k}, \Delta(\sum_{r=1}^kz_r),\rho_0(\sum_{r=1}^kz_r)\right)}{16\norm{e_n}}.\] 
Note $\norm{e_n}\geq \norm{b_n}\geq \norm{a_n}$ for all $n$.
 
 Claim 2:  Sequences $\tilde{x}_n= \sum_{r=1}^k z_r +s_na_n $ and $\tilde{y}_n=\sum_{r=1}^k z_r+s_ne_n$ achieve also the lower Lipschitz constant zero at $\sum_{r=1}^k z_k$.

 Note that $\max(\norm{s_na_n},\norm{s_ne_n})\le \frac{1}{16}$. 
Pass to subsequences of $(a_n)_n$ and $(e_n)_n$  so that both $\lim_{n \to \infty}s_na_n$ and $ \lim_{n \to \infty} s_ne_n$ converge.
Let $\alpha=\lim_{n \to \infty}s_na_n$ and $\delta=\lim_{n \to \infty} s_ne_n$. Notice $\delta \neq 0$.

 The limits
\[\lim_{n \to \infty}c_{r,n}=\lim_{n \to \infty}d_{r,n}=1,~\forall r \in [k]~\text{and } \lim_{n \to \infty} a_n= \lim_{n \to \infty} e_n=0 \]
imply that $\exists m_0 \in \N$ such that $\forall n \geq m_0$, and $\forall r \in [k]$
\begin{enumerate}
   \item  $|1-c_{r,n}|<\frac{1}{16k} \Delta(\sum_{r=1}^k z_r)$
   \item  $|1-d_{r,n}|<\frac{1}{16k} \Delta(\sum_{r=1}^k z_r)$
    \item  $|c_{r,n}-d_{r,n}|<\frac{1}{16k} \Delta(\sum_{r=1}^k z_r)$
    \item $\norm{a_n} <\frac{1}{16k} \Delta(\sum_{r=1}^k z_r)$
    \item $\norm{e_n} <\frac{1}{16k} \Delta(\sum_{r=1}^k z_r)$
\end{enumerate}
From \Cref{lemhelplow} part (1), \[\Delta(\sum_{r=1}^k c_{r,n} z_r) \geq \frac{1}{4} \Delta(\sum_{r=1}^k  z_r).\]
Also \[\max(\norm{a_n}, \norm{s_na_n})< \frac{1}{16} \Delta(\sum_{r=1}^k z_r) \le \frac{1}{4} \Delta(\sum_{r=1}^k c_{r,n} z_r) \] 
and
\[\max(\norm{e_n}, \norm{s_ne_n})< \frac{1}{16} \Delta(\sum_{r=1}^k z_r) \le \frac{1}{4} \Delta(\sum_{r=1}^k d_{r,n} z_r) \] 
So, for any $(i,j)\in S$
\begin{align*}
\hspace{-10mm} L^{i,j}(x_n)&=L^{i,j}(\sum_{r=1}^k c_{r,n}z_r+ a_{n})=L^{i,j}(\sum_{r=1}^k z_r+a_{n})\\&=L^{i,j}(\sum_{r=1}^k z_r+s_{n} a_{n})=
L^{i,j}(\tilde{x}_n)\subset L^{i,j}(\sum_{r=1}^k z_r)  =L^{i,j}(\sum_{r=1}^k c_{r,n}z_r).
\end{align*}

Where the second equality comes from \Cref{lemhelplow} part 3, third equality from \Cref{lemnest2} the fifth inclusion from \Cref{lemnest}, and the last equality from  \Cref{lemhelplow} part 2.

Similarly,
\begin{align*}
\hspace{-10mm} L^{i,j}(y_n)&=L^{i,j}(\sum_{r=1}^k d_{r,n}z_r+ b_{n})=L^{i,j}(\sum_{r=1}^k (1+d_{r,n}-c_{r,n}) z_r+ b_{n})\\&=L^{i,j}(\sum_{r=1}^k z_r+ e_{n}) =
L^{i,j}(\sum_{r=1}^k z_r+ s_ne_{n})=L^{i,j}(\tilde{y}_n)\\ &\subset L^{i,j}(\sum_{r=1}^k z_r)  =L^{i,j}(\sum_{r=1}^k c_{r,n}z_r).
\end{align*}

Therefore,
\begin{align*}
        0= &\lim_{n \to \infty}\frac{\norm{\Phi_{\bw,S}(x_n)-\Phi_{\bw,S}(y_n)}^2}{\dis([x_n],[y_n])^2}= \\
        =&\lim_{n \to \infty}\frac{ \sum_{i=1}^p \sum_{j \in S_i}\ip{w_i}{U_{g_{1,i,j}^{-1}}a_n -U_{g_{2,i,j}^{-1}} e_n}^2}{\norm{a_n-e_n}^2}\\
        =& \lim_{n \to \infty}\frac{ \sum_{i=1}^p \sum_{j \in S_i}\ip{w_i}{U_{g_{1,i,j}^{-1}}s_na_n -U_{g_{2,i,j}^{-1}} s_ne_n^2}}{\norm{s_na_n-s_ne_n}^2} \\
        = & \lim_{n \to \infty}\frac{\norm{\Phi_{\bw,S}(\tilde{x}_n)-\Phi_{\bw,S}(\tilde{y}_n)}^2}{\dis([\tilde{x}_n],[\tilde{y}_n])^2}.
    \end{align*}

where
\[{g_{1,i,j}} \in L^{i,j}(x_n) \text{ and } {g_{2,i,j}} \in L^{i,j}(y_n)\]
are chosen independent of $n$ by possibly passing to subsequences since $G$ is finite. 
So,
\[\Phi_{\bw,S}(\sum_{r=1}^k z_r+\alpha)-\Phi_{\bw,S}(\sum_{r=1}^k z_r+\delta)=0.\] 
Since $\hat{\Phi}_{\bw,S}$ is injective,
\[\sum_{r=1}^k z_r+\alpha \sim \sum_{r=1}^k z_r+\delta\]
Let $g_1\in G$ denote a group element that achieves this equivalence, i.e.
\[\sum_{r=1}^k z_r+\alpha = U_{g_1}(\sum_{r=1}^k z_r+\delta)\]
Note that $g_1 \in H(\sum_{r=1}^k z_r)$ because otherwise
 \begin{align*}
	0= \norm{\sum_{r=1}^k z_r+\alpha-U_{g_1}(\sum_{r=1}^k z_r)+U_{g_1}\delta)} = \norm{\sum_{r=1}^k z_r+\alpha-U_{g_1}(\sum_{r=1}^k z_r)+U_{g_1}\delta)} \\
	 \geq \norm{\sum_{r=1}^k z_r-U_{g_1}(\sum_{r=1}^k z_r)} - \norm{\alpha-U_{g_1}\delta} \geq 
  \rho_0(\sum_{r=1}^k z_r) - \norm{\alpha}-\norm{\delta} 
	 >0
	\end{align*}
 The last inequality comes from the fact that $\norm{\alpha}< \frac{1}{4}\rho_0(\sum_{r=1}^k z_r)$, and $\norm{\delta}< \frac{1}{4}\rho_0(\sum_{r=1}^k z_r)$. 

Additionally, $\alpha=U_{g_1}\delta$ because 
 \begin{align*}
	0= \norm{\sum_{r=1}^k z_r+a-U_{g_1}(\sum_{r=1}^k z_r)+U_{g_1}\delta)} 
	 =   \norm{a-U_{g_1}\delta}.
	\end{align*}

Claim 3: The two vectors $\alpha$ and $\delta$ are equal, $\alpha=\delta$.

We prove this claim by contradiction. 
Assume that $\alpha \neq \delta$.
From \Cref{lemalign}, $\exists M_0 \in   \N$ such that $\forall n\geq M_0$ 
\[\norm{s_na_n-s_ne_n} \le \norm{s_na_n-s_nU_{g_1}e_n}.\]
Therefore,
\[0<\norm{a-\delta} =\lim_{n \to \infty}\norm{s_na_n-s_ne_n} \le \lim_{n \to \infty}\norm{s_na_n-s_nU_{g_1}e_n} = 0.\]
We conclude that $\alpha=\delta\neq 0$.
\vspace{4mm}

Set $z_{k+1}=\alpha=\delta$. Together with sequences 
$\tilde{x}_n$ and $\tilde{y}_n$, they satisfy the assertions of part 2 of this Lemma.

\end{enumerate}
\end{proof}
\begin{rem}
Our construction produces $z_{k+1}$ that has norm equal to 
    
    $\frac{1}{16}\min\left(\norm{z_k}, \Delta(\sum_{r=1}^kz_r),\rho_0(\sum_{r=1}^kz_r)\right)$.
\end{rem}

\vspace{5mm}

Now we can complete the proof of \Cref{theormainstab}.

\begin{figure}
    \centering
\begin{tikzpicture}[tdplot_main_coords, scale = 0.8]

\coordinate (P) at ({5},{0},{5});
\coordinate (Q) at ({5+0.7},{0+0.5*sqrt(2)}, {5-0.7});


\shade[ball color = lightgray,
	opacity = 0.5
] (P) circle (2cm);
\shade[ball color = yellow,
	opacity = 0.5
] (Q) circle (0.75cm);


\draw[-stealth] (0,0,0) -- (4,0,0);
\draw[-stealth] (0,0,0) -- (0,4,0);
\draw[-stealth] (0,0,0) -- (0,0,4);
\draw[dashed, gray] (0,0,0) -- (-1,0,0);
\draw[dashed, gray] (0,0,0) -- (0,-1,0);

\draw[thick, -stealth] (0,0,0) -- (P) node[label=above:{$z_1$}]{};
\draw[thick, -stealth] (P) -- (Q) node[label=right:{$z_2$}]{};


\draw[fill = lightgray!50] (P) circle (0.5pt);

\end{tikzpicture}
\begin{tikzpicture}[tdplot_main_coords, scale = 0.8]

\coordinate (P) at ({5},{0},{5});
\coordinate (Q) at ({5+0.7},{0+0.5*sqrt(2)}, {5-0.7});

\coordinate (R) at ({5+0.7-0.088},{0+0.5*sqrt(2)+0.25}, {5-0.7+0.088});

\shade[ball color = lightgray,
	opacity = 0.5
] (P) circle (2cm);
\shade[ball color = yellow,
	opacity = 0.5
] (Q) circle (0.75cm);

\shade[ball color = red,
	opacity = 0.5
] (R) circle (0.125cm);

\draw[-stealth] (0,0,0) -- (4,0,0);
\draw[-stealth] (0,0,0) -- (0,4,0);
\draw[-stealth] (0,0,0) -- (0,0,4);
\draw[dashed, gray] (0,0,0) -- (-1,0,0);
\draw[dashed, gray] (0,0,0) -- (0,-1,0);

\draw[thick, -stealth] (0,0,0) -- (P) node[label=above:{$z_1$}]{};
\draw[thick, -stealth] (P) -- (Q) node[label=left:{$z_2$}]{};

\draw[thick, -stealth] (Q) -- (R) node[label=below:{$z_3$}]{};

\draw[fill = lightgray!50] (P) circle (0.5pt);

\end{tikzpicture}
    \caption{The robotic arm method.}
   \label{robarm}
\end{figure}

\begin{proof}
Starting from vector $z_1$ and the sequences $(x_n)_n$, $(y_n)_n$ observed in \Cref{lemupper} after $d-1$ steps of algorithmic construction of part (2) of \Cref{lemsub} we get $d$ non-zero vectors $\{z_1,\dots,z_d\}$  and a pair of sequences 
$(\tilde{x}_n)_n$, $(\tilde{y}_n)_n$
such that 

(i) $\ip{z_{i}}{z_j}=0,~\forall i,j \in [d], i \neq j$;

(ii) $\norm{z_{k+1}} \le \min(\frac{1}{4}\Delta(\sum_{r=1}^k z_r),\frac{1}{4}\norm{z_k}),~\forall k \in [d-1];$
and

(iii)
$\lim_{n \to \infty} \tilde{x}_n =\lim_{n \to \infty} \tilde{y}_n = \sum_{r=1}^{d} z_r$ and \[\lim_{n \to \infty}\frac{\norm{\Phi_{\bw,S}(\tilde{x}_n)-\Phi_{\bw,S}(\tilde{y}_n)}^2}{\dis([\tilde{x}_n],[\tilde{y}_n])^2}=0.\]

 Let $\tilde{x}_n=\sum_{r=1}^d l_{r,n}z_r$ and $\tilde{y}_n=\sum_{r=1}^d t_{r,n}z_r$.
 Notice that $\lim_{n \to \infty}l_{r,n}=\lim_{n \to \infty}t_{r,n}=1,~\forall r \in [d]$.
 
Recall that from part (2) of \Cref{lemhelplow} we have that $\exists M_0 \in \N$ such that $\forall n \geq M_0$ and $(i,j) \in S$
\[L^{i,j}(\sum_{r=1}^d l_{r,n}z_r)=L^{i,j}(\sum_{r=1}^d t_{r,n}z_r)=L^{i,j}(\sum_{r=1}^d z_r).\]
Then, for $g_{i,j} \in L^{i,j}(\sum_{r=1}^d z_r) $,

\begin{align*}
0&=\lim_{n \to \infty} \frac{\norm{\Phi_{\bw,S}(\tilde{x}_n) -\Phi_{\bw,S}(\tilde{y}_n)}^2}{\dis([x_n],[y_n])^2}\\
& =\lim_{n \to \infty} \frac{ \sum_{i=1}^p \sum_{j \in S_i}|\ip{U_{g_{i,j}} w_i}{\tilde{x}_n-\tilde{y}_n}|^2}{\dis([x_n],[y_n])^2}\\ 
&\geq \lim_{n \to \infty}\frac{ \sum_{i=1}^p \sum_{j \in S_i}|\ip{U_{g_{i,j}} w_i}{\sum_{r=1}^k(l_{r,n}-t_{r,n})z_r}|^2}{\norm{\sum_{r=1}^d(l_{r,n}-t_{r,n})z_r}^2}\\
&= \sum_{i=1}^p \sum_{j \in S_i}|\ip{U_{g_{i,j}} w_i}{\tilde{z}}|^2,
\end{align*}

where 
\[\tilde{z} =  \lim_{m \to \infty} \frac{\sum_{r=1}^d (l_{r,n_m}-d_{r,{n_m}})z_r}{\norm{\sum_{r=1}^d(l_{r,n_m}-t_{r,{n_m}})z_r}}\]
is a unit vector obtained as the limit of a convergent subsequence of the sequence of unit vectors
$\frac{\sum_{r=1}^d(l_{r,n}-t_{r,{n}})z_r}{\norm{\sum_{r=1}^d(l_{r,n}-t_{r,n})z_r}}$.
 Since the group $G$ is finite, we can find a positive number $\epsilon >0$ such that $\epsilon \norm{\tilde{z}} <\frac{1}{4}\Delta(\sum_{r=1}^dz_r)$ and 
$\sum_{r=1}^dz_r \nsim \sum_{r=1}^dz_r+\epsilon \tilde{z}$. In this case  
\[\Phi_{\bw,S}(\sum_{r=1}^dz_r)=\Phi_{\bw,S}(\sum_{r=1}^dz_r+\epsilon \tilde{z})\]
which contradict the injectivity property.
\Cref{theormainstab} is now proved.

\end{proof}

\section{Dimension reduction using linear maps}

In previous section we considered an embedding $\hat{\Phi}_{\bw,S}:\hat{\Vc}\rightarrow\R^m$ that is injective on the quotient space $\hat{\Vc}$. The dimension $m$ of the space $\R^m$ may be very large. In this section we show that the nonlinear map  $\hat{\Phi}_{\bw,S}$ can be further linearly processed into a smaller dimensional space while preserving injectivity and bi-Lipschitz properties. 

The idea of using dimension reduction linear maps goes back many years. The famous Johnson-Lindenstrauss Lemma \cite{johnson1986extensions} provides a nearly isometric projection for finite metric spaces. More recently, \cite{Cahill19} shows that a dimension reduction linear map preserves Lipschitz properties for translation invariant polynomial embeddings.

The first result that combines sorted co-orbits with linear maps was shown in \cite{balan2022permutation} in the case of the group of permutations $G=S_n$ acting by left multiplication on $n\times d$ real matrices, $\Vc=\R^{n\times d}$. In that particular case, the dimension of the intermediary co-orbit space was $n(1+(d-1)n!)$ while the final target space was shown to be $2nd=2\,dim(\Vc)$. The authors of \cite{dym2022low} show that intermediate embedding can be realized in dimension $(2nd+1)n$ instead of $n(1+(d-1)n!)$ with the final target space of dimension $2nd+1$. 
A more careful analysis of such an embedding by Matthias Wellershoff \cite{Matthias2023a} 
proved that $2nd+1$ can be replaced by $2nd-d$. In the Ph.D. thesis \cite{tsoukanis2024g}, one of the authors showed that the decrease of dimension from $2nd$ to $2nd-d$ 
represents the dimension of a certain joint eigenspace of this representation. 

The exact statement of our dimension reduction result 
is included in \Cref{dimred}.

To prove \Cref{dimred} we need first to construct a few objects.
Let $D : \Vc \times \Vc
\to \R^{m}$ be the nonlinear map $D(x,y)= \Phi_{\bw,S}(x)-\Phi_{\bw,S}(y)$.
Its range $E$ is defined by $E=\Ran(D) :=\{\Phi_{\bw,S}(x)-\Phi_{\bw,S}(y): x,y  \in \Vc
\}= \Ran(\Phi_{\bw,S})- \Ran(\Phi_{\bw,S})$.

Fix $g_1,\dots,g_N$ an enumeration of the group elements. Define $\lambda_{i,j}(x): \Vc \to \R$ by $\lambda_{i,j}(x) = \ip{U_{g_i}w_j}{x}$.
Notice that $\lambda_{i,j}$ is a linear map (unlike $\Phi_{i,j}$) and also that 
\begin{align*}
	\Phi_{\bw,S}(x)-\Phi_{\bw,S}(y) &=  [\lambda_{1,\nu_1(1)}(x)- \lambda_{1,\nu_{p+1}(1)}(y), \dots, \lambda_{1,\nu_1(m_1)}(x)- \lambda_{1,\nu_{p+1}(m_1)}(y),\\ &\dots,\lambda_{p,\nu_p(1)}(x)- \lambda_{p,\nu_{2p}(1)}(y), \dots, \lambda_{p,\nu_p(m_p)}(x)- \lambda_{p,\nu_{2p}(m_p)}(y) ]  
\end{align*}
for some permutations $\nu_1,...,\nu_{2p} \in S_N$ that may {\em depend} on $x$ and $y$. Let $m_j=|S_j|=|\{i\in[N]~,~(i,j)\in S\}|$ so that $m_1+\cdots+m_p=m$.  

Now, fix permutations $\pi_1,\ldots,\pi_{2p}\in S_N$ and let $L_{\pi_1, \dots,\pi_{2p}}:\Vc \times \Vc \to \R^{m} $ denote the {\em linear} map
\begin{align*}
	L_{\pi_1, \dots,\pi_{2p}}(x,y)=& [\lambda_{1,\pi_1(1)}(x)- \lambda_{1,\pi_{p+1}(1)}(y), \dots, \lambda_{1,\pi_1(m_1)}(x)- \lambda_{1,\pi_{p+1}(m_1)}(y),\\ &\dots,\lambda_{p,\pi_p(1)}(x)- \lambda_{p,\pi_{2p}(1)}(y), \dots, \lambda_{p,\pi_p(m_p)}(x)- \lambda_{p,\pi_{2p}(m_p)}(y) ]  
\end{align*}Define 
\[F = \cup_{\pi_1, \dots, \pi_{2p} \in S_{N}} \Ran(L_{\pi_1, \dots \pi_{2p}}).\]

Notice that $F$ is a finite union of linear subspaces and that $E \subset F$. For  fixed $\pi_1,\dots \pi_{2p}$ the map $(x,y)\mapsto L_{\pi_1,\dots , \pi_{2p}}(x,y)$ is linear in  $(x,y)$, $L_{\pi_1,\dots , \pi_{2p}}(v,v) = 0$ for all $v\in V_G$, and from the rank-nullity theorem we have \[\dim(\Ran(L_{\pi_1, \dots \pi_{2p}})) \le 2d-d_G.\]
\begin{lem}\label{lem:generic}
	Assume $r,s,m$ are non-negative integers so that $r+s \le m$. For any finite collection $\{F_a\,:\,a \in [T] \}$ of $T$ linear subspaces of $\R^m$ of dimension at most $s$, a generic $r$-dimensional linear subspace $K$ of $\R^m$ , satisfies $K \cap F_a =\{0\},\ \forall a \in [T]$. Here generic means open and dense with respect to Zarisky topology.
\end{lem}
\begin{proof}
	Let $\{v_1, \dots v_r\}$ be a spanning set for $K$, and $\{w_{1,a} \dots w_{M-r,a}\}$ be a linearly independent set of vectors
	such that $F_a \subset \Span\{w_{1,a}, \dots, w_{m-r,a} \}$. Then, $\Span\{v_1,\dots,v_r \}\cap\Span\{w_{1,a}, \dots, w_{m-r,a}\}=\{0\}$ if, and only if, the set $\{v_1,\dots,v_r,w_{1,a},\dots,w_{m-r,a}\}$ is linearly independent. Define $R_a(v_1, \dots v_r)= \det[v_1|\dots v_r|w_{1,a}|\dots w_{m-r,a}]$, and note that $R_a(v_1, \dots v_r)$ is a polynomial in $rm$ variables $v_1(1),\dots,v_1(M), \dots v_r(1),\dots,v_r(m)$. Hence,
	\begin{align*}
		K\cap F_a =\{0\},\quad \forall a \in [N]	
		&\iff		R_a(v_1, \dots v_r) \not=0,\quad \forall\ a \in [N]	\\
		&\iff		\prod_{a=1}^N R_{a}(v_1,\dots,v_r) \not= 0.
	\end{align*}	
	We conclude that \[\mathbb{U} =\left\{(v_1,\dots,v_r):\prod_{a=1}^NR_a(v_1,\dots,v_r) \not=0\right\}\] is an open set with respect to Zariski topology. In order to show that $\mathbb{U}$ is generic we have to find a $\{v_1,\dots,v_r\}$ such that $\prod_a R_a(v_1,\dots,v_r)\not=0$.	
 
 Let $W_a= \Span\{w_{1,a},\dots,w_{m-r,a}\}$. Notice that each $\Span(w_{1,a},\dots,w_{m-r,a}\}$ is a linear subspace of $\R^{m}$ of  dimension $m-r$. 
 If $r \geq 1$, each $W_a$ is a proper  subspace of $\R^m$. 
 
A generic $v_1\in\R^m$ satisfies $v_1 \not= 0$ and $v_1 \notin \cup_{a=1}^N W_a$. Replace each $W_a$ with $W_a^1= \Span(W_a, \{v_1\})$, subspaces of dimension $\dim(W_a^1) = \dim(W_a)+1=m-r+1$. If $\dim(W_a^1) < m$, repeat this process inductively and obtain $v_2, \dots, v_r$  until $\dim(W_a^r) = m$.
	The procedure produces a set of vectors $(v_1,\dots,v_r)$ that satisfy the condition $\prod_a R_a(v_1,\dots,v_r)\not=0$. Hence $\mathbb{U}\neq\emptyset$. This ends the proof of \Cref{lem:generic}.
 \end{proof}

Now we apply this lemma to derive the following corollary for our setup:

\begin{cor}\label{cor2}
Consider the $(N!)^{2p}$ linear maps
 $L_{\pi_1, \dots,\pi_{2p}}:\Vc \times \Vc \to \R^{m}$ introduced before. 
 Then for a generic $\ell:\R^m\rightarrow\R^{2d-d_G}$, 
    \[ ker(\ell)\bigcap \cup_{\pi_1,...,\pi_{2p}\in S_N} Ran(L_{\pi_1, \dots,\pi_{2p}}) = \{0\} \]
\end{cor}
\begin{proof}
If $m\leq 2d-d_G$ then the conclusion is satisfied for any full-rank $\ell$. Therefore assume $m>2d-d_G$.
 A  generic linear map $\ell: \R^{m} \to \R^{2d-D_G}$ is full-rank. Hence
 $\dim(\Ran(\ell))=2d-d_G$, and thus $\dim(\ker(\ell))= m-2d+d_G$. On the other hand, for a generic linear map $\ell$ 
 \Cref{lem:generic} with $r=m-2d+d_G$, $s=2d-d_G$, and $T=(N!)^{2p}$,
  implies
\[\ker(\ell) \cap \Ran(L_{\pi_1, \dots \pi_{2p}}) = \{0\} \]
for every $\pi_1,\ldots,\pi_{2p}\in S_N$.
\end{proof}

\begin{lem}\label{lem4.7}
Let $\{F_a\}_{a=1}^T$ be a finite collection of $r$-dimensional subspaces of $\R^m$, and $\ell: \R^m \to \R^s$ be a full-rank linear transformation with $m\geq s$. Let $Q_a$ denote the orthogonal projection onto the linear space $F_a$ and $Q_{\ell}$ denote the orthogonal projection onto 
 $\ker{\ell}$. Let $c_{a,\ell}=(1-\norm{Q_a Q_{\ell}}^2)^{1/2}$, and $c_{\ell}=\min_{a \in [T]} c_{a,\ell}$. Here $\norm{Q_a Q_{\ell}}$ denotes the 
 operator norm of $Q_aQ_\ell$, i.e., its largest singular value. 
Set $F=\cup_{a=1}^T F_a$. 
Suppose that $\ker(\ell) \cap F = \{0\}$.
 Then 
\begin{equation}
    \inf_{\substack{x \in F \\ \norm{x}=1}}\norm{\ell(x)}\geq c_{\ell} \sigma_{s}(\ell),
    \end{equation}
where $\sigma_{s}(\ell)$ is the smallest strictly positive singular value of $\ell$ (it is the $s^{th}$ singular value).
\end{lem}
\begin{proof}
	Notice that for each $a \in [T]$, the unit sphere of $F_a$ is a compact set. Thus
	\[\inf_{\substack{x \in F \\ \norm{x}=1}}\norm{\ell(x)}=\min_{\substack{x \in F \\ \norm{x}=1}}\norm{\ell(x)}=\norm{\ell(y_{\infty})}\]
	for some $y_{\infty} \in F_a \cap S^1(\R^m)$. 
	Let $y_{\infty}= \sum_{k=1}^m \gamma_k u_k$, where $u_j$ are the normalized right singular vectors of $\ell$ sorted by singular values $\sigma_1\geq \sigma_2\geq \cdots\geq \sigma_q > \sigma_{s+1}=\cdots=\sigma_m=0$. Notice that 
 $\sum_{k=1}^m \gamma_k^2=1$ and 
 $\sum_{k=1}^{s} \gamma_k^2 = 1-\norm{Q_\ell y_{\infty}}^2 \ge 1- \norm{Q_{a} Q_{\ell}}^2 \geq c_\ell^2$. Thus 
	\begin{align*}
		\norm{\ell(y_{\infty})}^2&= \norm{\sum_{k=1}^m \gamma_k \ell(u_k)}^2= \norm{\sum_{k=1}^s \gamma_{k} \ell(u_k)}^2\\& = \sum_{k=1}^s \gamma_k^2 \sigma_k^2 \geq c_{a,\ell}^2 \sigma_s(\ell)^2 \geq c_{\ell}^2 \sigma_s(\ell)^2
	\end{align*}
 which proves this Lemma.
\end{proof}

\begin{proof}[{\bf Proof \Cref{dimred}}]

Assume without loss of generality that $m\geq 2d-d_G$.

\Cref{cor2} shows that, a generic linear map $\ell:\R^m\rightarrow\R^{2d-d_G}$ satisfies $\ker(\ell)\cap \Ran(D) = \{0\}$. Thus, if $x,y\in\Vc$ so that $\Psi_{\bw,S,\ell}(x)=\Psi_{\bw,S,\ell}(y)$ then $\ell(D(x,y))=0$. Therefore $D(x,y)=0$. Since $\hat{\Phi}_{\bw,S}$ is injective it follows $x\sim y$.  Thus, $\hat{\Psi}_{\bw,S,\ell}$ is also injective.
 
From \Cref{theormainstab} we have that, if the map  $\hat{\Phi}_{\bw,S}$ is injective then it is also bi-Lipschitz. Let $a\leq b$ denote its bi-Lipschitz constants.
 
 Compositions of two Lipschitz maps is Lipschitz, hence $\Psi_{\bw,S,\ell}$ is Lipschitz. Furthermore, an upper Lipschitz constant of $\Psi_{\bw,S,\ell}$ is $\norm{\ell}b$, where $\norm{\ell}=\sigma_1(\ell)$ is the largest singular value of $\ell$.
 
 Finally from \Cref{cor2} and \Cref{lem4.7} with $r=m-2d+d_G$, $s=2d-d_G$, $T=(N!)^{2p}$, $F_a=Ran(L_{\pi_1, \dots,\pi_{2p}})$ we have that for a generic linear map $\ell$,  for all $x,y\in \Vc$,
 \begin{align*}
\norm{\Psi_{\bw,S,\ell}(x)-\Psi_{\bw,S,\ell}(y)} &= \norm{\ell(D(x,y))}\geq \\c_\ell \sigma_{2d-d_G}(\ell) \norm{D(x,y)}&\geq c_{\ell}\sigma_{2d-dG}(\ell)a \dis([x],[y])
 \end{align*} 
 where $a$ is the lower Lipschitz constant of $\hat{\Phi}_{\bw,S}$. Therefore the map $\hat{\Psi}_{\bw,S,\ell}$ is bi-Lipschitz with a lower Lipschitz constant $c_{\ell}\sigma_{2d-d_G}(\ell)a$.
\end{proof}

\begin{rem}
We proved that if $\hat{\Phi}_{\bw,S}$ is injective then for almost any linear map $\ell:\R^m\rightarrow\R^{2d}$,  $\hat{\Psi}_{\bw,S,\ell}$ is bi-Lipschitz. It remained an open question whether for any such nonlinear embedding $\Psi_{\bw,S,\ell}$, injectivity implies bi-Lipschitz.
However we settle this question positively into an upcoming joint paper with Matthias Wellershoff. Notice that, 
in general, if the map $f:X\rightarrow Y$ is bi-Lipschitz and the linear map $\ell:Y\rightarrow\R^q$ is so that $\ell\circ f$ is injective, then $\ell\circ f$ may not be bi-Lipschitz. Example: $f:\R\rightarrow\R^2$, $f(t)=(t,t^3)$, $\ell:\R^2\rightarrow\R$, $\ell(x,y)=y$.  
\end{rem}

\section{Universality of representation}

In previous sections we constructed the embedding  $\hat{\Phi}_{\bw,S}$ of the quotient space $\hat{\Vc}$ into an Euclidean space $\R^m$. In this section we prove \Cref{LIPTHEOR} and \Cref{cont} that show that when 
$\hat{\Phi}_{\bw,S}$ is injective, every continous or Lipschitz map   $f:\hat{\Vc} \to H$ factors through $\hat{\Phi}_{\bw,S}$.

The proof of \Cref{LIPTHEOR} is based on Kirszbraun extension theorem \cite{kirszbraun1934zusammenziehende} which is re-stated here for the reader's convenience:

\begin{theor}[Kirszbraun extension theorem\cite{kirszbraun1934zusammenziehende}]\label{Kirszbraun}
Let $E\subset H_1$ be an arbitrary subset of a Hilbert space $H_1$ and $f:E \mapsto H_2$ be a Lipschitz function to another Hilbert space $H_2$. Then there exists an extension $F:H_1\rightarrow H_2$ of $f$ to the entire space $H_1$ that has the same Lipschitz constant as the original function $f$.
\end{theor}

\begin{proof}[Proof of \Cref{LIPTHEOR}]
\mbox{}

\begin{enumerate}
    \item   Let $t:\Phi_{\bw,S}(\Vc) \mapsto H$ be defined by $t(\Phi_{\bw,S}(x))=F([x])$. Denote $u=\Phi_{\bw,S}(x)$ and $ v=\Phi_{\bw,S}(y)$.
    Then, \begin{align*}
        \norm{t(u)-t(v)}= \norm{F([x])-F([y])} \le \Lip(F) \dis([x],[y]) \le \frac{1}{a} \Lip(F) \norm{u-v}.
    \end{align*}
    By Kirszbraun extension theorem we have that there exists $ T:\R^m \to H$, such that
    \begin{enumerate}
        \item $T\vert_{\Phi_{\bw,S}(\Vc)} =t$  
        \item $\Lip(T)= \Lip(t)$
    \end{enumerate}
    Therefore,  $F=T \circ \Phi_{\bw,S}$ and $\Lip(T)\le \frac{1}{a} \Lip(F)$. 
    
   \item Part 2 is straightforward. Let $x,y \in \Vc$ then 
   \begin{align*}
       \norm{F(x)-F(y)} &= \norm{T\circ \Phi_{\bw,S}(x)-T\circ \Phi_{\bw,S}(y)}\\
       &\le \Lip(T) \norm{\Phi_{\bw,S}(x)-\Phi_{\bw,S}(y)} \\ &\le \Lip(T) a \norm{x-y}.
   \end{align*}
   
\end{enumerate}
\end{proof}

The second universality result \Cref{cont} 
applies to the class of continuous functions (instead of Lipschitz functions).


The proof of \Cref{cont} follows from the following extension  
of Tietze's theorem  \cite{urysohn1925machtigkeit,dugundji1951extension}.

\begin{theor}[Dugundji-Tietze \cite{dugundji1951extension}]\label{tietze}
 Let $X$ be a metric space and $A$ a closed subset of $X$. Let $L$ be a locally convex topological vector space. Given $f:A \mapsto L$ a continuous map, there exists a continuous extension 
 $F: X \mapsto L$ such as $F(X)$ is a subset of the convex hull  of $f(A)$.
\end{theor}

\begin{proof}[Proof of \Cref{cont}]  

Let $S= \Phi_{\bw,S}(\Vc)\subset\R^m$. Note that $S$ is a closed set as $\hat{\Phi}_{\bw,S}$ is bi-Lipschitz and $\hat{\Vc}$ is complete. Let $t:S \mapsto L$ be defined by $t(\Phi_{\bw,S}(x))= F(x)$, for all $x\in\Vc$. Note $t$ is continuous since $\hat{\Phi}_{\bw,S}$ is bi-Lipschitz on $\hat{\Vc}$. 
By \Cref{tietze} there exists a continuous extension $T:\R^m\rightarrow L$ of $t$ that satisfies the convex hull property.

\ignore{

Note that, $\Phi_{\bw,S}^{-1}: S \mapsto \Vc$ is also continuous and that 
\[\dis([t(u)],[t(v)])=\dis([F(x)],[F(y)]).\] 

\begin{lem}
\begin{enumerate}
    \item Assume that $\hat{F}$ is continuous. Then $F: \Vc \mapsto \R^D$, where $F([x])=\hat{F}([x])$ is continuous on $(V,\ip{\cdot}{\cdot})$.
    \item  Let $F:\Vc \mapsto \R^m$, such that $F$ is continuous and $G$-invariant.
    Then $\hat{F}:\hat{V} \mapsto \R^m$, where $\hat{F}([x])=F(x)$ is continuous.
\end{enumerate}
\end{lem}

\begin{proof}
    Let $\Pi: \Vc \mapsto \hat{\Vc}$, $\Pi(x)=[x]$.
   Notice that $\dis([\Pi(x)],[\Pi(x)])=\min_{\substack{u \in [x]\\v \in [y]}}\norm{u-v} \le \norm{x-y}$.
\begin{enumerate}
    \item  $\Pi$ is continuous and $F=\hat{F} \circ \Pi$ therefore, $F$ is continuous.
    \item Notice that because $G$ is compact we can extract a convergent subsequence such that $\tilde{y}_n \in [y_n]$ such that $\tilde{y}_n \to x$.
 \end{enumerate} 
\end{proof}
}

\end{proof}
\printbibliography
\end{document}